\documentclass[11pt, reqno]{amsart}
\usepackage{fullpage}
\usepackage{amsmath,amsthm,amssymb,enumerate}
\usepackage[all]{xy}
\usepackage{array}
\usepackage{float}
\usepackage{graphicx}
\usepackage{tikz}

\usepackage[dvipsnames]{xcolor}
\definecolor{citation}{rgb}{0,.40,.80}
\usepackage[colorlinks,linkcolor=citation,citecolor=citation,urlcolor=citation]{hyperref}
\usepackage{cleveref}

\crefname{table}{Table}{Tables}
\Crefname{table}{Table}{Tables}

\newcommand \C {\mathbb C}

\newcommand \F {\mathbb F}
\newcommand \fF {\mathcal F}
\newcommand \gG {\mathcal G}

\newcommand \I {\mathcal I}

\renewcommand \L {\mathcal L}

\newcommand \fm {\mathfrak m}
\renewcommand \O {\mathcal O}
\newcommand \N {\mathbb N}

\renewcommand \P {\mathbb P}

\newcommand \Z {\mathbb Z}

\newcommand\xra{\xrightarrow}

\newcommand\dA{\hat{A}}

\DeclareMathOperator \alb {alb}

\DeclareMathOperator \sEnd {\mathcal{E}nd}

\DeclareMathOperator \sExt {\mathcal{E}xt}
\DeclareMathOperator \id {id}

\DeclareMathOperator \sHom {\mathcal{H}om}
\DeclareMathOperator \Hilb {Hilb}
\DeclareMathOperator \Supp {Supp}

\DeclareMathOperator \length {length}

\DeclareMathOperator \mult {mult}
\DeclareMathOperator \Pic {Pic}
\DeclareMathOperator \NS {NS}
 
\DeclareMathOperator \Spec {Spec}

\DeclareMathOperator{\gr}{gr}

\newtheorem {thm} {Theorem}[section]
\newtheorem {cor} [thm] {Corollary}
\newtheorem {lemma} [thm] {Lemma}
\newtheorem {prop} [thm] {Proposition}

\theoremstyle{definition}
\newtheorem {defn} [thm] {Definition}
\newtheorem{question} {Question}
\newtheorem {example}[thm] {Example}
\newtheorem {rmk}[thm] {Remark}
\crefalias{lemma}{lemma}     
\crefalias{prop}{proposition} 
\crefalias{rmk}{remark}    
\crefalias{cor}{corollary}    
\crefname{rmk}{Remark}{Remarks}
\Crefname{rmk}{Remark}{Remarks}

\numberwithin{equation}{section}
\numberwithin{table}{section}

\makeatletter
\newcommand{\xleftrightarrow}[2][]{\ext@arrow 3359\leftrightarrowfill@{#1}{#2}}
\newcommand{\dashto}[2][]{\ext@arrow 0359\rightarrowfill@@{#1}{#2}}
\newcommand{\xdashleftarrow}[2][]{\ext@arrow 3095\leftarrowfill@@{#1}{#2}}
\newcommand{\xdashleftrightarrow}[2][]{\ext@arrow 3359\leftrightarrowfill@@{#1}{#2}}
\def\rightarrowfill@@{\arrowfill@@\relax\relbar\rightarrow}
\def\leftarrowfill@@{\arrowfill@@\leftarrow\relbar\relax}
\def\leftrightarrowfill@@{\arrowfill@@\leftarrow\relbar\rightarrow}
\def\arrowfill@@#1#2#3#4{%
  $\m@th\thickmuskip0mu\medmuskip\thickmuskip\thinmuskip\thickmuskip
   \relax#4#1
   \xleaders\hbox{$#4#2$}\hfill
   #3$%
}
\makeatother

\newcommand{\obot}{\mathbin{\vcenter{\hbox{
  \begin{tikzpicture}[x=0.84ex,y=0.84ex,line width=0.06ex,baseline=-0.7ex]
    \draw (0,0) circle (0.9);
    \draw (-0.8,-0.4) -- (0.8,-0.4);
    \draw (0,-0.4) -- (0,0.9);
  \end{tikzpicture}
}}}}

\newcommand{\qformtable}[8]{%
    \makebox[0pt][l]{\rule[-1em]{0pt}{2.5em}}
    \vspace{0em}
    \begin{tikzpicture}[baseline=(current bounding box.center), scale=0.4]
        \fill[#1] (0,1) circle (0.4); 
        \fill[#2] (1,1) circle (0.4); 
        \fill[#3] (0,0) circle (0.4); 
        \fill[#4] (1,0) circle (0.4); 
        \node at (2.2,0.5) {$\obot$};
        \fill[#5] (3.4,1) circle (0.4); 
        \fill[#6] (4.4,1) circle (0.4); 
        \fill[#7] (3.4,0) circle (0.4); 
        \fill[#8] (4.4,0) circle (0.4); 
    \end{tikzpicture}%
}
                  
\begin{document}

\author{Katrina Honigs}
\author{Graham McDonald}
\author{Peter M. McDonald}
\address{Department of Mathematics\\
Simon Fraser University\\
8888 University Drive\\
Burnaby, BC, V5A 1S6\\
Canada}

\title{An approach to curves in abelian surfaces using Fourier--Mukai and quadratic forms}

\date{}

\begin{abstract} 
It was proven by Yoshioka \cite[Prop.~3.5]{Yoshioka} that 
given a complex abelian surface $A$ 
of Picard rank $1$
whose primitive polarization is  
non-principal of type 
$(1,d)$, there is an isomorphism $\Psi:\Hilb^d_A\times\hat{A}\to M_{\hat{A}}(0,\hat{l},-1)$ where $\Hilb^d_A$ is the Hilbert scheme of lenth-$d$ subschemes of $A$ and 
$M_{\hat{A}}(0,\hat{l},-1)$ is
a moduli space of Gieseker-stable sheaves on the dual abelian surface. 
Specifically, $M_{\hat{A}}(0,\hat{l},-1)$ parametrizes
rank $1$ torsion-free sheaves with Euler characteristic $-1$ that are supported on a curve 
whose 
N\'eron--Severi class is dual to that of the polarization on $A$. 
The Fourier--Mukai transform is a crucial component of $\Psi$.

In this paper, we use the isomorphism $\Psi$ 
to deduce information 
about curves on abelian surfaces. 
In the case that $Z\in \Hilb^d_A$ is symmetric, i.e., fixed by the inverse map $\iota$ on $A$, 
we use quadratic forms to compute information about the sheaf $\Psi(Z)$ and its supporting curve. It was recently shown by Knutsen and Lelli-Chiesa \cite{KLC} that 
any singularity on a curve of geometric genus $2$ contained in a general 
$(d_1,d_2)$-polarized abelian surface must have multiplicity at most $6$, among other constraints.
In contrast, we demonstrate there are curves with singularities of arbitrarily high multiplicity contained in 
general $(1,d)$-polarized abelian surfaces for sufficiently large $d$.
Furthermore, 
we identify the isolated fixed points of $\iota$ in 
acting on the variety of Kummer type 
$K_{\hat{A}}(0,\hat{l},-1)$ when $d=4$. Along the way, we prove 
some structural results on symmetric line bundles, showing 
that if $d$ is even, the dual of an odd line bundle is odd. 
\end{abstract}

\maketitle

Recently, significant advances have been made in characterizing the specific types of curves that arise within linear systems of general 
$(d_1,d_2)$-polarized abelian surfaces.
For example, the number of hyperelliptic curves that such a linear system contains is finite \cite{Pirola} and 
in \cite{Bryan}, the authors find their virtual counts for each genus.
When the polarizing line bundle is primitive and 
the abelian surface satisifies a condition that is expected to hold generically,
the counts of \cite{Bryan} are exact.

The smooth hyperelliptic curves in general primitively polarized abelian surfaces have now all been explicitly constructed. They have genus $g=2$, $3$, $4$, or $5$ and are contained in 
$(1,g-1)$-polarized abelian surfaces, respectively \cite[Theorem~2.8]{BorowkaOrtega}. The constructions of the genus $2$ and $3$ cases are classically known, and the genus $4$ and $5$ cases were completed more recently in \cite{BSgenus4curves} and \cite{BorowkaOrtega}.

Curves of genus $2$ in abelian surfaces have several special qualities.
The total number of genus $2$ curves can be counted relatively simply using kernels of isogenies \cite{Debarre}; this count has been used to calculate the Euler characteristic of generalized Kummer varieties. 
These curves are known to have ordinary singularities \cite{LangeSernesi}, and recently it was proven 
that such curves have only nodal singularities unless $4|d_2$, in which case the singularities may be ordinary triple, $4$-tuple, or $6$-tuple points, all of which occur \cite{KLC}. However, relatively little is known about curves of higher genus, both in terms of explicit constructions and constraints on what sort of singularities may occur.

In this paper, we take a novel approach toward analyzing such curves,
under the constraint that the abelian surface $A$ has Picard rank $1$ and 
the N\'eron--Severi class of the curve is a primitive, non-principal polarization ($d_1=1$, $d_2\geq 2$).
Under these assumptions, there is an isomorphism of moduli spaces that we call $\Psi$, the essential part of which is a specific case of 
\cite[Prop.~3.5]{Yoshioka}, that identifies the Hilbert scheme of $d$-points on $A$ with, roughly speaking, subschemes of compactified relative Jacobians of 
the dual linear systems 
on the dual abelian surface $\hat{A}$.
This isomorphism is well-established; the restriction 
of $\Psi$
to the generalized Kummer variety $K_{d-1}A$ 
endows it with a Lagrangian fibration, which comes from mapping sheaves in the relative Jacobian to their supporting curves in $|\hat{L}|$ (see \cite{Gulbrandsen}).

Given a $d$-point $Z\in\Hilb^d_A$, 
it is not clear how to determine concrete geometric properties 
of its image under $\Psi$ in general, but the following
structural result in this direction suggests that $\Psi$ nevertheless 
provides a natural framework in which to study the supporting curves of the sheaves in its codomain. 
Translation-invariance of $Z$, a straightforward property on the Hilbert scheme side of $\Psi$, has deep geometric consequences for the supporting curve $C_Z$ of $\Psi(Z)$, including determining its irreducible component in the Severi variety. 

\begin{thm}\label{structure_thm_intro}
Let $(A,L)$ be a general $(1,d)$-polarized abelian surface with $d\geq 2$. Let
$Z$ be a $d$-point on~$A$. 
\begin{enumerate}[(a)]
\item If $Z$ is invariant under translation by $x\in A\setminus K(L)$, $t_{\phi_L(x)}$ is a nontrival endomorphism on~$C_Z$.
\end{enumerate}
Let $G_Z\leq K(L)$ be the subgroup of all points in $K(L)$ under which $Z$ is translation invariant.
\begin{enumerate}[(a)]\setcounter{enumi}{1}
\item The geometric genus of $C_Z$ is at most $\frac{d}{|G_Z|}+1$, and is $2$ if and only if $|G_Z|=d$. 
\item When the geometric genus of $C_Z$ is greater than or equal to $3$, 
 $G_Z$ determines its component of the Severi variety (cf.~Zahariuc \cite{Zahariuc}).
\end{enumerate}  
\end{thm}

The isomorphism $\Psi$ depends on the choice of a polarizing line bundle on the abelian surface. If we choose this line bundle 
to be symmetric, then $\Psi$ commutes with 
the inverse group law map $\iota:A\to A$, $a\mapsto -a$.
This compatibility 
allows us to use the eigenvalues of $\iota^*$ to analyze 
$\Psi(Z)$ when $Z$ is symmetric ($\iota(Z)=Z$).

\begin{thm}\label{eigen}
If $Z\in K_{d-1}A\subseteq\Hilb^d_A$ is symmetric, then $C_Z\in |\hat{L}|$. Furthermore:
\begin{enumerate}[(a)]
\item The eigensystem $|\hat{L}|^{\pm}$ containing $C_Z$ is
specified by the determinant of the action of $\iota$ on a vector space we associate to $Z$.
\item For each $2$-torsion point $x\in\hat{A}[2]$, there is a lower bound on the dimension of 
$\Psi(Z)\otimes k(x)$ that we can compute
with
a quadratic form. 
\end{enumerate}  
\end{thm}

Since the multiplicity of $C_Z$ at a point $x$ is bounded below by the dimension of $\Psi(Z)\otimes k(x)$. 
We analyze families of examples $Z$ using the above theorem to prove the following result.

\begin{thm}\label{highmult}
Let $m\in\N$.
There are infinitely many $d\in\N$ so that every 
$(1,d)$-polarized abelian surface of Picard rank $1$ contains a curve 
with even (resp.\ odd) multiplicity at least $m$ in its polarizing linear system.
\end{thm}

Related to our exploration of the isomorphism $\Psi$, 
we also show that for a polarization of type $(1,d)$,
the dual of the unique odd polarizing line bundle is odd \Cref{oddodd}.

When we define $\Psi$ using a symmetric line bundle, its restriction 
to varieties of  Kummer type $\Psi:K_{d-1}A\to K_{\hat{A}}(0,\hat{l},-1)$
respects the fixed loci of $\iota^*$, a symplectic automorphism. 
Symplectic automorphisms are significant in the study of hyperk\"ahler varieties, and the number and dimensions of the  
components of these particular fixed loci was determined in \cite{KMO}. 
Although components of higher dimension have been shown to play a significant role in, for instance, cohomology \cite{HasTsc}, the isolated fixed points 
remain unexplored. 
In order to shed light on this structure in $K_{\hat{A}}(0,\hat{l},-1)$, where we expect to find curves with interesting symmetries associated with these points, 
we apply our methods to the case when $d=4$.
We obtain the following analysis:

\begin{thm}\label{140}
Let $Z\in K_3A$ be one of the $140$  isolated fixed points of the involution $\iota$.
\begin{enumerate}[(a)]
\item The curve $C_Z$ is one of the following:
\begin{enumerate}[(i)]
\item a smooth hyperelliptic curve ($64$ cases),
\item a curve of geometric genus $3$ with two nodes ($72$ cases),
\item a curve of geometric genus $2$ with an ordinary triple point singularity ($4$ cases).
\end{enumerate}
\item The lower bound given by \Cref{eigen}(b) is sharp. 
\item We are able to specify each sheaf $\Psi(Z)$: if $C_Z$ is smooth, we 
write down its divisor class, and if $C_Z$ is singular, it is the 
pushforward of a line bundle on the normalization $\widetilde{C}_Z$ whose divisor class we can write down. 
\end{enumerate}
\end{thm}

In \Cref{Psi}, we introduce background on abelian varieties, polarizations, and quadratic forms.
In \Cref{sec.moduli}, we examine the isomorphism $\Psi:\Hilb^d_A\times\hat{A}\to M_{\hat{A}}(0,\hat{l},-1)$ and its restriction to Albanese fibers $K_{d-1}A\to K_{\hat{A}}(0,\hat{l},-1)$. Then, we discuss some general properties that $\Psi$ satisfies and prove \Cref{structure_thm_intro}. 
In \Cref{sec.compute}, we prove \Cref{eigen} and \Cref{oddodd}, 
and discuss further methods for analyzing sheaves $\Psi(Z)$.
We then have two sections of applications: we produce examples of curves with high-multiplicity singularities in \Cref{sec.nonnodal}, proving \Cref{highmult}, and analyze $140$ isolated fixed points when $d=4$ in \Cref{sec.140ex}, proving \Cref{140}.
Finally, we discuss the context of these examples and 
pose questions on the general nature of the isolated fixed points of $\iota^*$ acting on $K_{\hat{A}}(0,\hat{l},-1)$ and curves in linear systems of abelian surfaces.
Eigenvalues for \Cref{sec.140ex}
computed with \textsf{Sage} \cite{sage}
code \cite{HMM}
are placed in an appendix.

\section*{Acknowledgements}
K.H.\ was supported by an NSERC Discovery Grant.
G.M.\ was supported by an NSERC Canada Graduate Research Scholarship -- Master's (CGRS M).
P.M.\ was supported by the Pacific Institute for Mathematical Sciences and a grant from the Simons Foundation International [SFI-MPS-T-Institutes-00020822, OY].

\section{Background on abelian surfaces}\label{Psi}

\subsection{Abelian varieties, polarizations, Fourier--Mukai}\label{bkgrd}
Given an abelian variety $A$, we denote the 
identity element by $e_A$, and the
inverse group law map by $\iota:A\to A$, $\iota(a)=-a$. 
We refer to any object invariant under the appropriate action of $\iota$ as \emph{symmetric}.

Let $\hat{A}$ be the dual abelian variety.
Given a point $x\in A$, we denote the corresponding line bundle on $\hat{A}$ by $P_x\in \Pic^0(\hat{A})$. 
Given a line bundle $L$ on $A$, we denote its N\'eron--Severi class by $l$ and let $\phi_L:A\to \hat{A}$ be given by $\phi_L(a)=t_a^*L\otimes L^{-1}$ with $K(L):=\ker\phi_L$. We denote the linear system of $L$ by $|L|:=\P(H^0(A,L))$. The eigenspaces of $|L|$ and $H^0(A,L)$ under the action of $\iota^*$
with eigenvalue $\pm1$  are denoted 
$|L|^{\pm}$ and $H^0(A,L)^{\pm}$.

Given a polarizing line bundle $L$ on an abelian surface $A$, 
its type $(d_1,d_2)$ is the unique pair of integers so that $d_1|d_2$ and $K(L)\simeq (\Z/d_1\Z)^2\times (\Z/d_2\Z)^2$. 
We call the associated Weil pairing on the $2$-torsion $e_L:A[2]\times A[2]\to \{\pm1\}$, $e_L:=e_2(-,\phi_L(-))$ where $e_2$ is the pairing between $A[2]$ and $\hat{A}[2]$.

In the case where $L$ is of $(1,d)$-type and $d$ is even,
we 
name the bases of $A[2]\simeq\F_2^{\oplus 4}$ and $\hat{A}[2]\simeq \F_2^{\oplus 4}$ as
$\{w_1,k_1,w_2,k_2\}$ and
$\{x_1,y_1,x_2,y_2\}$, with 
$k_3:=k_1+k_2$, etc. 
We set the convention that the 
matrix for $e_L$ and $\phi_L$ is:
\[
\left(
\begin{smallmatrix}
0&0&1&0\\
0&0&0&0\\
1&0&0&0\\
0&0&0&0
\end{smallmatrix}\right).
\]
Then, $K(L)\cap \hat{A}[2]=\langle k_1,k_2\rangle$ and 
$\phi_L(w_1)=x_2$ and $\phi_L(w_2)=x_1$, hence
$\phi_L(A[2])=\langle x_1,x_2\rangle$. 

We refer to the Fourier--Mukai transform with with kernel the Poincar\'e bundle as \emph{the Fourier--Mukai transform}. This functor relating sheaves on an abelian surface to those on its dual is an essential ingredient in our constructions.
It is the exact functor
mapping from the bounded derived category of coherent sheaves on $A$ to those on $\hat{A}$
as follows:
\begin{align}
  \Phi:D(A)&\to D(\hat{A})\\
  \fF&\mapsto  Rp_{2*} (Lp_1^*(\fF)\otimes^{L} P_A)
\end{align}
The maps $p_1,p_2$ are the projections from $A\times \hat{A}$ to its factors. The functor $\Phi$ is an equivalence \cite{Mukai} and has the following properties.

\begin{lemma}[{{\cite[Ex.~2.6, (3.1), Thm.\ 2.2]{Mukai}}}]\label[lemma]{FMfacts}
Let $x\in A$, $y\in \hat{A}$, and let $P_x\in \Pic^0(A)$, $P_y\in \Pic^0(\hat{A})$ be the corresponding line bundles.
For any $\fF\in D(A)$, we have:
\begin{align}
\quad\quad\quad\quad\Phi(F\otimes P_y)&=t_y^*\Phi( F), &
\Phi(t_x^*F)&=P_{-x}\otimes\Phi(F),\quad\quad\quad\quad
\label{tensortranslate}
  \\
  \Phi(P_y)&= k(-y)[-g],\label{linebundleskyscraper} &
                                                         \Phi(k(x))&=P_x,\\
  \Phi\circ\Phi&\cong \iota^*\circ [-g] \rlap{$\text{ where } g=\dim(A)$.}
\label{compose}                     
\end{align}
\end{lemma}

\subsection{Symmetric polarizing line bundles and quadratic forms}\label{sym.sec}

The symmetric line bundles with N\'eron--Severi class $l$ form a principal homogeneous space over the $2$-torsion line bundles $\Pic^0(A)[2]$, hence there are $2^{2g}$ of them, where $g=\dim(A)$. 

Let $L$ be symmetric a polarizing line bundle of $(1,d)$-type on an abelian surface $A$. We may choose an isomorphism $f:\iota^*L\to L$ that is normalized so  the fiber over $e_A$ is the identity map. For any $a\in A[2]$, the restriction of $f$ to the fiber of $L$ over $a$ is given by multiplication by a nonzero scalar $q_L(a)\in \{\pm 1\}$. The map $q_L:A[2]\to \{\pm 1\}$ given by $a\mapsto q_L(a)$ is a quadratic form over the Weil pairing on $A[2]$. 

The quadratic forms  $q:A[2]\to \{\pm 1\}$ have a natural torsor structure over $\hat{A}[2]$, where $x\in \hat{A}[2]$ acts on a quadratic form $q$ to produce the form $q_x$ as follows:
\begin{equation}\label{qfaction}
q_x(a):=q(a)e_2(a,x)
\quad\text{for $a\in A[2]$}.
\end{equation}
If we act on $q_L$ by $x$, we get the quadratic form associated with $L\otimes P_x$.

We distinguish different types of symmetric line bundles $L\in \Pic^l(A)$ when $l$ is an ample class.
If $d$ is odd, then $L$ must 
have at least one \textit{symmetric theta structure} (sts). However, when $d$ is even, having an sts is
equivalent to 
one of $h^0(L)^{+}$ or $h^0(L)^{-}$ having maximal dimension $\frac12 h^0(L)+1$ (see \cite[\S6.9,\S6.10(10)]{BirkenhakeLange}). 
If $L$ has at least one sts, 
then $L$ is \emph{even} or \emph{odd} according to the parity of $q_L$. If $q_L$ is even (odd), it has a (non)trivial valuation at the majority of the points in $A[2]$. If $d$ is odd, every symmetric line bundle has a sts and there are $10$ even line bundles and $6$ odd line bundles. If $d$ is even, there are $12$ line bundles with no sts, $4$ with sts, and among those, $3$ even and $1$ odd. 

We show the set of all possible quadratic forms 
in the case where $d$ is even 
in \Cref{tab:qform_labels}, choosing the convention 
that $e_{\hat{A}}$ corresponds to the unique odd line bundle.

We aggregate the various qualities of symmetric line bundles $L$ in the following proposition. The base loci of $|L|^{\pm}$ are contained in $A[2]$ and denoted by $A[2]^{\mp}$, named for the value that $q_L$ takes at these base points (cf.~\Cref{prop.eigencurve.mult}).

\begin{prop}[{{\cite[Prop.\ 4.7.5, Exercise~6.10(10)]{BirkenhakeLange}}}, see also \cite{BolognesiMassarenti}]\label[prop]{dimcount}

Let $A$ be an abelian surface that is $(1,d)$-polarized by a symmetric line bundle $L$. \\
If $d$ is odd, then $L$ has sts, and:
\begin{align*}
&\text{$L$ even:}\quad
& h^0(A,L)^+&=\tfrac{d+1}{2},\ \#A[2]^-=6,\quad 
  &h^0(A,L)^-&=\tfrac{d-1}{2}, \ \#A[2]^+=10.\\[1ex]
&\text{$L$ odd:}\quad
& h^0(A,L)^+&=\tfrac{d-1}{2},\ \#A[2]^-=10,
  &h^0(A,L)^-&=\tfrac{d+1}{2},\ \#A[2]^+=6.
\end{align*}
If $d$ is even: 
\begin{align*} 
&\text{$L$ sts, even: }\quad &h^0(A,L)^+&=\tfrac{d}{2}+1,\ \# A[2]^-=4,\quad\quad
&h^0(A,L)^-&=\tfrac{d}{2}-1,\ \#A[2]^+=12.\\
&\text{$L$ sts, odd: }\quad &h^0(A,L)^+&=\tfrac{d}{2}-1,\ \# A[2]^-=12,
& h^0(A,L)^-&=\tfrac{d}{2}+1,\ \#A[2]^+=4.\\
&\text{$L$ no sts: }
  \quad &h^0(A,L)^+&=\tfrac{d}{2},\ \# A[2]^-=8,
&h^0(A,L)^-&=\tfrac{d}{2},\ \#A[2]^+=8.  
\end{align*}
\end{prop}

Furthermore, the following result shows that an eigencurve $C\in |L|^{\pm}$
 must pass through the points in the base locus of its eigensystem with odd multiplicity and other $2$-torsion points with even (possibly $0$) multiplicity.

\begin{prop}\label[prop]{prop.eigencurve.mult}
Let $a\in A[2]$ and $C\in|L|^{\pm}$ be an eigencurve of $\iota^*$ with eigenvalue $\lambda_C=\pm 1$.
Then, $\mult_aC$ is even 
if $q_L(a)=\lambda_C$  and odd if $q_L(a)=-\lambda_C$.
\end{prop}  

\begin{proof}
Let $s\in\Gamma(A,L)$ be a section whose vanishing is $C$ and 
fix an isomorphism 
$f\colon\iota^*L\to L$ 
where the fiber over $e_A$ is the identity map.
Let $a\in A[2]$ and $s_a\in L_a$ the restriction of $s$.
We know that $s_a\in\fm_a^rL_a\setminus\fm_a^{r+1}L_a$ for some $r\geq0$ (where $\fm_a^0=\O_{A,a}$); this $r$ is the multiplicity of $C$ at $a$. As noted in \cite[\S13.2]{Polishchukbook}, the action of $\iota^*$ on $\fm_a^rL_a/\fm_a^{r+1}L_a$, and hence on $s_a$, is given by multiplication by 
\[
(-1)^rq_L(a)f|_e=(-1)^rq_L(a).
\]
Since the eigenvalue of $\iota^*$ acting on $s$, and hence on $s_a$, is $\lambda_C$, we have $\lambda_C=(-1)^rq_L(a)$.
\end{proof}

\begin{table}[h]
    \centering
    \begin{tabular}{|c||c|c|}
        \hline
        
        & Quadratic Form & \begin{tabular}{c|c} $e$ & $k_1$\\
        \hline $k_2$ & $k_3$\end{tabular}$\;\obot\;$\begin{tabular}{c|c} $e$ & $w_1$\\ \hline $w_2$ & $w_3$\end{tabular}\\

        \hline
        \hline

        odd with sts & $q_{e_{\!\hat{A}}}$ & \qformtable{lightgray}{lightgray}{lightgray}{lightgray}{lightgray}{pink}{pink}{pink}\\

        \hline

        even with sts & $q_{x_1}$ & \qformtable{lightgray}{lightgray}{lightgray}{lightgray}{lightgray}{pink}{lightgray}{lightgray}\\

        \cline{2-3}

        & $q_{x_2}$ & \qformtable{lightgray}{lightgray}{lightgray}{lightgray}{lightgray}{lightgray}{pink}{lightgray}\\

        \cline{2-3}

        & $q_{x_3}$ & \qformtable{lightgray}{lightgray}{lightgray}{lightgray}{lightgray}{lightgray}{lightgray}{pink}\\

        \hline

        no sts & $q_{y_1}$ & \qformtable{lightgray}{lightgray}{pink}{pink}{lightgray}{pink}{pink}{pink}\\

        \cline{2-3}

        & $q_{x_1+y_1}$ & \qformtable{lightgray}{lightgray}{pink}{pink}{lightgray}{pink}{lightgray}{lightgray}\\

        \cline{2-3}

        & $q_{x_2+y_1}$ & \qformtable{lightgray}{lightgray}{pink}{pink}{lightgray}{lightgray}{pink}{lightgray}\\

        \cline{2-3}

        & $q_{x_3+y_1}$ & \qformtable{lightgray}{lightgray}{pink}{pink}{lightgray}{lightgray}{lightgray}{pink}\\

        \cline{2-3}

        & $q_{y_2}$ & \qformtable{lightgray}{pink}{lightgray}{pink}{lightgray}{pink}{pink}{pink}\\

        \cline{2-3}

        & $q_{x_1+y_2}$ & \qformtable{lightgray}{pink}{lightgray}{pink}{lightgray}{pink}{lightgray}{lightgray}\\

        \cline{2-3}

        & $q_{x_2+y_2}$ & \qformtable{lightgray}{pink}{lightgray}{pink}{lightgray}{lightgray}{pink}{lightgray}\\

        \cline{2-3}

        & $q_{x_3+y_2}$ & \qformtable{lightgray}{pink}{lightgray}{pink}{lightgray}{lightgray}{lightgray}{pink}\\

        \cline{2-3}

        & $q_{y_3}$ & \qformtable{lightgray}{pink}{pink}{lightgray}{lightgray}{pink}{pink}{pink}\\

        \cline{2-3}

        & $q_{x_1+y_3}$ & \qformtable{lightgray}{pink}{pink}{lightgray}{lightgray}{pink}{lightgray}{lightgray}\\

        \cline{2-3}

        & $q_{x_2+y_3}$ & \qformtable{lightgray}{pink}{pink}{lightgray}{lightgray}{lightgray}{pink}{lightgray}\\

        \cline{2-3}

        & $q_{x_3+y_3}$ & \qformtable{lightgray}{pink}{pink}{lightgray}{lightgray}{lightgray}{lightgray}{pink}\\

        \hline

    \end{tabular}
\vspace*{2ex}
    \caption{Quadratic forms on $A[2]$ for $d$ even.}
    \label{tab:qform_labels}
\end{table}

\subsection{Dual polarizations}

Let $A$ be an abelian variety with a polarizing line bundle $L$.
Since $L$ is ample its index (cf.~\cite[\S16]{Mumford}) is $0$, hence the complex $\Phi(L)$ has its only nonzero term in degree~$0$.

\begin{defn}
We define the \emph{dual line bundle} of $L$ on $\hat{A}$ to be $\hat{L}:=\det(\Phi(L))^{-1}$.
\end{defn}

In \cite{BLdual}, Birkenhake and Lange show that 
$\hat{L}$ is polarizing and refer to its N\'eron--Severi class, which we denote by $\hat{l}$, as the \emph{dual polarization} to $l$.

We assemble some useful facts about $\hat{L}$ in the following lemma. 
We prove some further results in this vein in \Cref{oddodd}, using the machinery of our isomorphism of moduli spaces.

\begin{lemma}\label[lemma]{PhiL_lemma}
Let $A$ be an abelian surface and $L$ a polarizing line bundle of $(1,d)$-type.
\begin{enumerate}[{(a)}]
\item $\phi_{\hat{L}}\circ \phi_L=d\cdot \id_A$ and $\phi_L\circ\phi_{\hat{L}}=d\cdot \id_{\hat{A}}$, which uniquely characterizes $\hat{l}$.
\item $\det(\Phi(L^{-1}))$ is inverse to $\det(\Phi(L))$.
\item The map $\Pic^l(A)\to \Pic^{\hat{l}}(\hat{A})$, 
$N\mapsto \hat{N}$, is $\phi_{\hat{L}}$-equivariant. 
For each $x\in \hat{A}$,
$(L\otimes P_x)^{\wedge}\simeq \hat{L}\otimes\phi_{\hat{L}}(x)$.
\item If $L$ is symmetric, then $\hat{L}$ is symmetric. 
\end{enumerate}
\end{lemma}

\begin{proof}
(a) is proven in condition $(1)$ of the main theorem of \cite{BLdual}.

(b) Since $L$ is a $(1,d)$-polarization, there is an isogeny $\pi:A\to B$ 
of degree $d$ 
to a principally polarized abelian surface $(B,P)$ so that $\pi^*P\simeq L$ and $\pi$ factors through $\phi_L$. 
Let $\tau$ be the isogeny so that $\tau\circ\pi=\phi_L$. By \cite[Prop.~3.11(4)]{Mukai}, $\Phi(L)\simeq \tau_*P^{-1}$. 
Since $L^{-1}$ is anti-ample, its index is $2=\dim(A)$. Thus, the only nonzero term in the complex $\Phi(L^{-1})$ is in degree $2$ and, again by \cite[Prop.~3.11(4)]{Mukai}, it is isomorphic to 
$\tau_*P$. 

Defining the norm map $N_{\tau}:\Pic(B)\to \Pic(\hat{A})$ as in 
\cite[Tag~0BCX]{stacks-project}, by Grothendieck--Riemann--Roch we have the following formula:
$$\det(\tau_*P)\simeq N_{\tau}(P)\otimes \det(\tau_*\O_B).$$

By \cite[II\S7, p.~72]{Mumford}, the pushforward of the structure sheaf along an isogeny is a sum of line bundles indexed by the character group of the kernel of an isogeny.
So, letting $K:=\ker{\tau}$, we have:
$$\tau_*\O_B\simeq \oplus_{\alpha\in\widehat{K}} L_{\alpha}.$$
Since $\pi$ has degree $d$, the kernel of $\pi$ is Lagrangian 
(i.e.\ maximal isotropic)
with respect to the Weil pairing on $A$ given by $L$, and hence isomorphic to its character group. 
Because $B$ is principally polarized, there is an isomorphism between $K$ and $\ker{\hat{\pi}}$, so $K$ inherits this favorable property, and thus $\widehat{K}\simeq K$. 
This means the group structure of $K$ is compatible with that of $\Pic(\hat{A})$, and 
so $\det(\tau_*\O_B)$ is isomorphic to the line bundle $L_{\beta}$ corresponding to $\sum_{\alpha\in\widehat{K}} \alpha$, which must be $2$-torsion, so 
$L_{\beta}=L_{\beta}^{-1}$.

We now have 
$\det(\tau_*P)\simeq N_{\tau}(P)\otimes L_{\beta}$, 
and similarly $\det(\tau_*P^{-1})\simeq N_{\tau}(P^{-1})\otimes L_{\beta}$. 
Because $N_{\tau}$ is a group homomorphism, 
$N_{\tau}(P^{-1})=N_{\tau}(P)^{-1}$, and we have 
$\det(\tau_*P^{-1})=\det(\tau_*P)^{-1}$.

(c) Let $x\in \hat{A}$. By \eqref{linebundleskyscraper} we have:
$
\det(\Phi(L\otimes P_x))^{-1}\simeq \det(t_x^*\Phi(L))^{-1}
\simeq t_x^*\hat{L}\simeq \phi_{\hat{L}}(x)\otimes \hat{L}$.

(d) By \eqref{compose}, $\Phi$ commutes with $\iota^*$.
\end{proof}

\section{Definition of $\Psi$ and proof of \Cref{structure_thm_intro}}\label{sec.moduli}

In \Cref{sec.Psi}, we work through the background needed to  define the isomorphism $\Psi:\Hilb^d_A\times\hat{A}\to M_{\hat{A}}(0,\hat{l},-1)$ and its restriction to Albanese fibers. In \Cref{invariance}, we will give some general properties of $\Psi$, including the proof of \Cref{structure_thm_intro}.

\subsection{Definition of the isomorphism $\Psi$ and restriction to Kummer fibers}\label{sec.Psi}
\subsubsection{Moduli of stable sheaves and definition of $\Psi$}
Let $A$ be an abelian surface with a polarization $H$ and a Mukai vector $v=(r,l,s)$.
The moduli space $M_A(v)$ parametrizes
$H$-stable sheaves on $A$ with Mukai vector $v$, that is, rank $r$, N\'eron-Severi class of the determinantal line bundle $l$, and Euler characteristic $s$.
If $v$  is
primitive and positive where $v^2:=l^2-2rs\geq 0$ and $H$ is $v$-generic, then
the moduli space of stable sheaves $M_A(v)$ is a non-empty, smooth, projective and irreducible variety of dimension $v^2+2$.
We refer the interested reader to the papers \cite{Mukaisymp} and \cite{Yoshioka}, or the summary in \cite{Frei_Honigs} for more details on this construction and hypotheses.

Given a 
$(1,d)$-polarizing line bundle $L$ on $A$ so that $d>1$ and $\NS(A)=\Z l$,
by work of Yoshioka \cite[Prop.~3.5]{Yoshioka} (see also \cite{Gulbrandsen}),
there is an isomorphism
$\Psi:\Hilb^d_A\times \hat{A}\xra{\sim} M_{\hat{A}}(0,\hat{l},-1)$, which is given by the following composition:
\begin{align}
  \Psi:  \Hilb^d_A\times \hat{A} &\xra{\sim} M_A(1,0,-d) \xra{\sim} M_A(1,l,0)  \xra{\sim}
                M_{\hat{A}}(0,\hat{l},-1)
\\\notag                                                          
  (Z,M)&\rlap{${}\mapsto I_{Z}\otimes M$}
       \phantom{\xra{\sim} K_A(1,0,-d)}\,\,
       \llap{$\fF$}\rlap{${}\mapsto \fF\otimes L$}
       \phantom{\xra{\sim} K_A(1,l,0)}\,\,\,\,\,\,
\llap{$\gG$}{}\mapsto \Phi(\gG)[1] 
\end{align}
The leftmost isomorphism sends a subscheme of length $d$ and a line bundle to the ideal sheaf of the subscheme tensored with the line bundle.
Tensoring with the line bundle $L$ gives the middle isomorphism.
The morphism on the right is given by the action of the Fourier--Mukai transform~$\Phi$. The images
under $\Phi$
are, \textit{a priori}, complexes of sheaves, but they are 
supported only in index $1$, so we shift the indexing to give a well-defined map to sheaves.

\subsubsection{Restriction of $\Psi$ to Kummer fibers}
The variety $K_A(v)$ is defined to be a fiber of the Albanese morphism on $M_A(v)$ and is of Kummer type. All such fibers are isomorphic, but
$\iota$ acts differently on these fibers, for instance having different distributions of fixed points, so we must track them carefully.
In this section we will 
show that $\Psi$ restricts to an isomorphism on 
Albanese fibers and
discuss how $\Psi$ interacts with these fibers.

When $v^2 \geq 2$,
the Albanese variety of a moduli space $M_A(v)$ 
is $A\times \hat{A}$ \cite[Theorem~0.1]{Yoshioka}. 
For convenience, we will use the morphism to the Albanese torsor:
\begin{align}\label{alb}
\alb: M_A(v)&\to \Pic^l(A)\times \Pic^{-\hat{l}}({\hat{A}})
\\ \notag
\fF&\mapsto (\det(\fF),\det(\Phi(\fF))).
\end{align}

In the following proposition, we will show that 
$\Psi$ respects the fibers of the Albanese morphisms by examining how the Albanese acts on its domain and codomain.

\begin{lemma}\label[lemma]{albhilb}
\begin{enumerate}[(a)]
\item   The Albanese morphism \eqref{alb}
  acts as follows on $\Hilb^d_A\times \hat{A}$:
\begin{align*}
\alb:\Hilb^d_A\times \hat{A}&\to \Pic^0(A)\times\Pic^0(\hat{A})\\
(Z,M)  &\mapsto (M, -\Sigma Z)
\end{align*}    
where $\Sigma Z$ is given by summing $Z$ in the group law of $A$; if a subscheme has length $n$ and its support is one point $p\in A$, then its sum is $n\cdot p$.
\item The Albanese morphism \eqref{alb}
acts as follows on the image of $(Z,M)\in \Hilb^d_A\times \hat{A}$ in 
$M_{\dA}(0,\hat{l},-1)$:
\begin{align*}
\alb: M_{\hat{A}}(0,\hat{l},-1)&\to \Pic^{\hat{l}}(\dA)\times\Pic^{-l}(A)\\
  \Psi(Z,M)  &\mapsto (t_m^*\hat{L} \otimes P_{\Sigma Z}, M\otimes\iota^*L^{-1}),
\end{align*}
where $m\in \hat{A}$ is the point corresponding to the line bundle $M\in \Pic^0(A)$.
\end{enumerate}
\end{lemma}

The above proposition implies that the map $\Psi$ respects the fibers of the Albanese morphism because the image of $\Psi(Z,M)$ under the Albanese morphism only depends on the image of $(Z,M)$ under the Albanese morphism.

\begin{proof}
(a) We use the identification $\Hilb^d_A\times \hat{A}\simeq M(1,0,-d)$ where $(Z,M)$ maps to $I_Z\otimes M$.
Let $I_Z$ be the ideal sheaf of $Z$.
We consider the short exact sequence
\begin{equation}\label{idealses}
0\to I_Z\to \O_A \to \O_A/I_Z\to 0.
\end{equation}
Thus, $\det(I_Z)\otimes \det(\O_A/I_Z)\simeq \det(\O_A)$. 
Since the support of $\O_A/I_Z$ is finite,
$\det(\O_A/I_Z)$ is trivial, and hence $\det(I_Z)$ is trivial. Thus $\det(I_Z\otimes M)\simeq M$.

Furthermore, we have $\det(\Phi(I_Z))\otimes \det(\Phi(\O_A/I_Z))\simeq \det(\Phi(\O_A))$.
By \eqref{linebundleskyscraper}, $\Phi(\O_A)$ is a shifted skyscraper sheaf and hence its determinant is trivial.

By \eqref{linebundleskyscraper}, for any $x\in A$, $\det(\Phi(k(x)))=P_x$. We may argue by induction on length that for any length $n$ skyscraper sheaf $\fF$ whose support is $x$, $\det(\Phi(\fF))\simeq P_x^{\otimes n}$: for instance if $\fF$ has length~$2$, it fits into a short exact sequence with $k(x)$ at both ends. Extending these ideas, we may write 
$\det(\Phi(\O_A/I_Z))$ as $\Sigma Z$ in $A\simeq\Pic^0(\hat{A})$.
Thus $\det(\Phi(I_Z))$ is $-\Sigma Z$.

Using \eqref{tensortranslate}, we have $\det(\Phi(I_Z\otimes M))\simeq 
t_m^*\det(\Phi(I_Z))$, where $m\in \hat{A}$ corresponds to $M$.
Since $\det(\Phi(I_Z))\in \Pic^0(\hat{A})$, it is invariant under translation, and hence $\det(\Phi(I_Z\otimes M))\simeq\det(\Phi(I_Z))$. 

(b) Let $(Z,M)\in \Hilb^d_A\times \hat{A}$. We use \Cref{FMfacts} to evaluate the Albanese map at $\Psi(Z,M)$:
\begin{align*}
  \det(\Psi(Z,M))&=\det(\Phi(I_Z\otimes M\otimes L)[1])\\
&\simeq t_m^*\det(\Phi(I_Z\otimes L))^{-1}
\end{align*}
Tensoring the short exact sequence \eqref{idealses} with $L$
and noting $\O_A/I_Z\otimes L\simeq \O_A/I_Z$, we have:
\[
\det(\Phi(L))\simeq\det(\Phi(I_Z\otimes L)\otimes\det(\Phi(\O_A/I_Z)).
\]
As shown in the proof of part (a), $\det(\Phi(\O_A/I_Z))\simeq P_{\Sigma Z}$. Thus:
\[
  \det(\Psi(Z,M))\simeq
                   t_m^*(\det(\Phi(L))\otimes P_{-\Sigma Z} )^{-1}
\simeq
t_m^*(\det(\Phi(L))^{-1} \otimes P_{\Sigma Z}
\simeq t_m^*\hat{L} \otimes P_{\Sigma Z}
  .\]
For the other term, we have:
\begin{align*}
\det(\Phi\circ \Psi(Z,M))&=\det(\Phi\circ\Phi(I_Z\otimes M\otimes L)[1])\\
                         &\simeq\det(\iota^*(I_Z\otimes M\otimes L)[-1])\\
&\simeq\iota^*(M\otimes L))^{-1}\simeq M\otimes \iota^*L^{-1}.\qedhere
\end{align*}
\end{proof}  

\begin{rmk}
 We define $K_{d-1}A\simeq K_A(1,0,-d)$ to be the Albanese fiber over $(\O_A,\O_{\hat{A}})$. Its image under $\Psi$ is the Albanese fiber of $M_{\hat{A}}(0,\hat{l},-1)$ over $(\hat{L},\iota^*L^{-1})$. We refer to this preferred fiber as $K_{\hat{A}}(0,\hat{l},-1)$. 
\end{rmk}

The moduli space $M_{\hat{A}}(0,\hat{l},-1)$
parametrizes (pushforwards to $\hat{A}$ of) rank $1$ torsion-free sheaves
with Euler characteristic $-1$
on curves in $\hat{A}$ that define a linear system with N\'eron--Severi class $\hat{l}$.
We now analyze its Albanese fibers in those terms.

Given an element $\fF\in M_{\hat{A}}(0,\hat{l},-1)$, the first term of $\alb(\fF)$ gives the linear system containing the supporting curve of $\fF$. 
Thus,
the preimage of the Albanese morphism \eqref{alb}
on $M_{\hat{A}}(0,\hat{l},-1)$  over
$\{\hat{L}\}\times \Pic^{-{l}}(A)$
is the relative compactified Jacobian
$\smash{\overline{\mathrm{Pic}}^{-1}_{\mathcal{C}/|\hat{L}|}}$, 
where $\mathcal{C}$ is the tautological family of curves in~$|\hat{L}|$
and we index by Euler characteristic. The elements of the compactified Jacobian are rank $1$ torsion-free sheaves supported on curves in the indicated linear system.

We now examine the significance of the second term of $\alb(\fF)$. Note that a rank one torsion-free sheaf on a curve with Gorenstein singularities can always be written as $I_Z\otimes\O_C(D)$ where $I_Z\subseteq\O_C$ is an ideal sheaf of a finite subscheme of $C$
and $D$ is an effective Cartier divisor. 
We also write $D$ for the subscheme with ideal sheaf $\O_C(-D)$.
The following lemma shows that the second term 
of the Albanese map is
the sum $\Sigma D -\Sigma Z$ where each point in $D$ and in $Z$ is multiplied by the length of $\O_C(D)$ or $I_Z$ at that point, respectively.

\begin{lemma}\label[lemma]{detPhi}
Let $C$ be a curve in the linear system $|\hat{L}|$ 
on $\hat{A}$
and let $a\colon C\to\hat{A}$ be the inclusion map. Consider the rank-one torsion free sheaf $\fF=I_Z\otimes\O_C(D)$ supported on $C$, where $I_Z\subseteq\O_C$ is an ideal sheaf of a finite length subscheme $Z\subset C$ and $D$ is an effective Cartier divisor. Then
  \[\det(\Phi(a_*\fF))\simeq \det(\Phi(\hat{L}))\otimes P_w\]
where $w=\sum D-\sum Z$ is the point in $\hat{A}$ given by summing under the group law. In particular, $\det(\Phi(a_*\O_C))\simeq \det(\Phi(\hat{L}))$.
\end{lemma}

\begin{proof}
We begin with the case where $\fF=\O_C$. There is a short exact sequence for the ideal sheaf of any such curve $C$: 
\[
0\to I_C\to \O_{\hat{A}} \to a_*\O_C \to 0.
\]
Noting that $I_C\simeq \hat{L}^{-1}$, we conclude by applying $\det(\Phi(-))$ to the short exact sequence:
\[
\det(\Phi(a_*\O_C))\simeq \det(\Phi(\hat{L}^{-1}))^{-1}.
\]
By \Cref{PhiL_lemma}(b),
$\det(\Phi(\hat{L}^{-1}))^{-1}\simeq \det(\Phi(\hat{L}))$.

Now consider $\fF=\O_C(D)$ for $D$ an effective Cartier divisor. We proceed by induction on the degree of $\O_C(D)$. Suppose $\O_C(D)$ has degree $1$ so that there is a short exact sequence
\[0\to \O_C(-D)\to\O_C\to k(p)\to0\]
for some closed point $p\in C$. Applying $\sHom(-,\O_C)$ we get the short exact sequence 
\[
0\to \O_C\to \O_C(D)\to \sExt^1(k(p),\O_C)\to 0.
\]
We have $\sExt^1(k(p),\O_C)\simeq k(p)$, as it is supported only at $p$ and $C$ is Gorenstein. Applying $\det(\Phi(a_*(-)))$ to the short exact sequence, we find that
$\det(\Phi(a_*(\O_C(D)))\simeq
\det(\Phi(\hat{L}))
\otimes P_{p}$. Now suppose $\O_C(D)$ is degree $d$. Then there is some $D'$ of degree $d-1$ such that there is a short exact sequence 
\[0\to\O_C(D')\to\O_C(D)\to k(q)\to0\]
for some closed point $q\in C$. Applying $\det(\Phi(a_*(-)))$ and invoking the inductive hypothesis gives $\det(\Phi(a_*\O_C(D)))\cong \det(\Phi(\hat{L}))\otimes P_{\sum D}$.

Now consider $\fF=I_Z\otimes\O_C(D)$. We proceed by induction on the length of $Z$, so suppose $Z$ is a single closed point $p\in C$. Consider the short exact sequence
\[0\to I_Z\otimes\O_C(D)\to\O_C(D)\to k(p)\otimes\O_C(D)\to0.\]
Applying $\det(\Phi(a_*(-)))$ and noting that $k(p)\otimes\O_C(D)\cong k(p)$ gives $\det(\Phi(a_*(I_Z\otimes\O_C(D))))\cong \det(\Phi(\hat{L}))\otimes P_{\sum D-p}$. Now suppose $Z$ has length $d$ and consider $Z'$ of length $d-1$ so that $I_{Z'}/I_Z\cong k(q)$ for some close point $q\in C$. Then consider the short exact sequence
\[0\to I_Z\otimes\O_C(D)\to I_{Z'}\otimes \O_C(D)\to k(q)\otimes\O_C(D)\to0.\]
Applying $\det(\Phi(a_*(-)))$ and invoking the induction hypothesis gives $\det(\Phi(a_*(I_Z\otimes\O_C(D))))\cong \det(\Phi(\hat{L}))\otimes P_{\sum D-\sum Z}$.    
\end{proof}  

\begin{rmk}\label[remark]{detPhidual}
The statement of \Cref{detPhi} also holds if $D$ isn't effective: $\O_C(D)$ is still a rank one torsion-free sheaf on $C$ and therefore $\O_C(D)\cong \O_C(-D_1)\otimes\O_C(D_2)$ for $D_1$, $D_2$ effective Cartier divisors. Taking this further, if $\fF=I_Z^\vee\otimes\O_C(D)$ where $D$ is any Cartier divisor, then the same argument shows $
  \det(\Phi(a_*\fF))\simeq\det(\Phi(\hat{L}))\otimes P_w$ where $w=\sum D+\sum Z$ simply by considering the exact sequence
  \[0\to I_{Z'}\otimes\O_C(D)\to I_Z\otimes\O_C(D)\to k(p)\otimes O_C(D)\to 0.\]
\end{rmk}

\begin{rmk}
Combining the results of Lemmas~\ref{albhilb} and \ref{detPhi}, we note that the image of $Z\in\Hilb^d_A$ under $\Psi$ is in the Albanese fiber of $M_{\hat{A}}(0,\hat{l},-1)$
over $(\hat{L}\otimes P_{\Sigma Z}, \iota^*L^{-1})$, so the linear system may vary depending on $\Sigma Z$ but the value $w$ is constant for any choice of $Z$.

The preimage of the relative compactified Jacobian
$\smash{\overline{\mathrm{Pic}}^{-1}_{\mathcal{C}/|\hat{L}|}}$ in $\Hilb^d(A)\times \hat{A}$ consists of pairs $(Z,M)$ where $\phi_{\hat{L}}(M)\simeq P_{-\Sigma Z}$.
\end{rmk}

\subsection{Proof of \Cref{structure_thm_intro}}\label{invariance}

We begin by examining how the image of a $d$-point under isomorphism $\Psi$ behaves if we translate the $d$-point.

\begin{lemma}\label[lemma]{lem.trten}
  Let $Z_1$ and $Z_2$ be $d$-points on $A$.
  Let $x\in A$. The $d$-point $Z_2$ is the translation of $Z_1$ by $x$ (i.e.\ $t_x(Z_1)=Z_2$)
  if and only if the following relation is satisfied:
  \begin{equation}\label{translatetensor}
    \Psi(Z_2)\simeq P_x\otimes t_{\phi_L(x)}^*\Psi(Z_1).\end{equation}
\end{lemma}

\begin{proof}
First, note that $t_x(Z_1)=Z_2$ if and only if $t_{-x}^*I_{Z_1}\simeq I_{Z_2}$.

Now, we assume that $t_x(Z_1)=Z_2$, which implies that 
$t_{-x}^*(I_{Z_1}\otimes L)\simeq I_{Z_2}\otimes t_{-x}^*L
\simeq I_{Z_2}\otimes L\otimes t_{-x}^*L \otimes L^{-1}
$. 
Applying $\Phi$ to the equation and using \eqref{tensortranslate} to simplify, we have:
\[
P_x\otimes\Psi(Z_1)
\simeq
t_{\phi_L(-x)}^*(\Psi(Z_2)),
\]
yielding \eqref{translatetensor}.

Conversely, assume that \eqref{translatetensor} holds. By using the definition of $\Psi$ and \eqref{tensortranslate}, we have:
\[
\Phi(t_x^*(I_{Z_2}\otimes L))\simeq \Phi(I_{Z_1}\otimes L\otimes (t_x^*L\otimes L^{-1})).
\]
Because $\Phi$ is an equivalence, the arguments in the above expression are isomorphic. Rearranging shows that $t_{-x}^*I_{Z_1}\simeq I_{Z_2}$, and we are done.
\end{proof}

\Cref{lem.trten} immediately implies the following consequences:

\begin{cor}[\Cref{structure_thm_intro}(a)]\label[cor]{rmk.trten} 
Let $Z$, $Z'$ be $d$-points on $A$. 
\begin{enumerate}[(a)]
\item If $Z'$ is a translation of $Z$, then the supporting curve $C_{Z'}$ of $\Psi(Z')$ is a translation of $C_{Z}$. 
\item Let $T_{Z}=\{x\in A\mid t_x(Z)=Z\}$
and $G_Z= K(L)\cap T_Z$. 
\begin{enumerate}[(i)]
\item If $x\in G_Z$, then $\Psi(Z)\simeq P_x\otimes \Psi(Z)$ and, by \cite[Lemma~2.1]{Beauville}, the line bundle $P_x$ pulls back trivially to $\Pic^0(\widetilde{C})$.
\item The elements of $\phi_L(T_Z)\simeq T_Z/G_Z$ are endomorphisms of $C_Z$.
\end{enumerate}
\end{enumerate}
\end{cor}

We will need the following result from \cite[17.2]{Polishchukbook} in order to prove the remainder of \Cref{structure_thm_intro}, which we note applies to singular curves.

\begin{lemma}[{\cite[Prop.~17.3]{Polishchukbook}}]\label[lemma]{curve.inc}
Let $a:C\to A$ be a nonconstant morphism from a curve to an abelian variety.
Let $a_*:\Pic(C)\to A$ be the morphism sending the formal sum $D=\sum n_i[p_i]$ to the sum
$\sum n_ia(p_i)$
in the group law of $A$.
For any line bundle $N$ on $C$,
the  composition
\begin{equation}
\Pic^0(A)\xra{a^*} \Pic^0(C)\xra{a_*} A
\end{equation}
is equal to the morphism $\phi_{L'}$, where $L':=\det(\Phi(a_*N))$. 
\end{lemma}

\begin{rmk}\label[remark]{rmk.curveinc}
(a) Consider the situation where $a:C\to A$ is the immersion of a curve in the linear system $|L|$ or is the composition of such an immersion with the normalization of such a curve. The pushforward of a line bundle on the normalization of a curve  is a rank $1$ torsion-free sheaf on a curve. By \Cref{detPhi} and the definition of the dual line bundle, the N\'eron--Severi class of $\det(\Phi(a_*N))$ in this setting is $-\hat{l}$, and 
$a_*\circ a^*=-\phi_{\hat{L}}$.

(b) As a consequence, for any $P_x\in \Pic^0(A)$, its restriction to $C$ 
is some rank $1$ torsion free sheaf $I_Z\otimes \O_C(D)$. Then $\Sigma D-\Sigma Z=0$ if and only if $x\in K(L)$.
\end{rmk}

\begin{thm}[\Cref{structure_thm_intro}(b),(c)]\label{genus2etc}
Let $(A,L)$ be a general $(1,d)$-polarized abelian surface. Let
$Z$ be a $d$-point on $A$. Let $G_Z\leq K(L)$ be the subgroup of all points in $K(L)$ under which $Z$ is translation invariant, and $C_Z$ be the supporting curve of $\Psi(Z)$ in $\hat{A}$.

\begin{enumerate}[(a)]
\item The geometric genus of $C_Z$ is at most $\frac{d}{|G_Z|}+1$, and is $2$ if and only if $|G_Z|=d$, which is the largest possible value of $|G_Z|$.

\item When the geometric genus of $C_Z$ is greater than or equal to $3$, 
  the group $G_Z$ determines which component of the Severi variety (in the sense of Zahariuc \cite{Zahariuc}) contains $C_Z$.
\end{enumerate}  
\end{thm}

We introduced the assumption that $(A,L)$ is general in order to use
the results of \cite{Debarre} and \cite{Zahariuc}. 

\begin{proof} First, if $Z$ is invariant under translations by elements of $G_Z$, then $G_Z$ acts freely on the supporting points of $Z$, of which there are at most $d$. Thus $|G_Z|\leq d$.

Let $a:\widetilde{C}_Z\to \hat{A}$ be the map given by composing the normalization of $C_Z$ with the immersion of $C_Z$ into $\hat{A}$.
By \cite[Lemma~2.1]{Beauville} and \Cref{rmk.trten}, 
the kernel of $a^*:\Pic^0(\hat{A})\to J\widetilde{C}_Z$ consists of
$\{P_x\mid x\in G_Z\}\simeq G_Z$.
Since $a^*$ factors as $A\to A/G_Z\to J\widetilde{C}_Z$,
and $A/G_Z\to J\widetilde{C}_Z$ must be injective,
the dual $a_*: J\widetilde{C}_Z \to \hat{A}$ factors sharply through
the isogeny
$A'\xra{\theta} \hat{A}$ 
where $A':=(A/G_Z){\!\hat{\phantom{a}}}$ and $\theta$ is dual to the quotient $A\to A/G_Z$.
Hence, the morphism $a$ also factors sharply though $\theta$.

As discussed in \Cref{sec.Psi}, 
$C_Z\subseteq\hat{A}$, the image of $\widetilde{C}_Z$ in $A$,  is contained in the linear system 
of a $(1,d)$-polarization 
whose N\'eron--Severi class is $\hat{l}$.
By \cite[Prop.~4.3]{BLmoduli}, 
the pullback of the principal polarization on $J\widetilde{C}_Z$ to $A$ has type $(1,d)$ and N\'eron--Severi class $l$. The linear system containing the image of $\widetilde{C}_Z$ in $A'$ is the pushforward of a polarization of $(1,d)$-type, so must have $(1,\frac{d}{|G|})$-type (see \cite[Lemma~1.3]{HonigsSarmah}).
Thus, $\frac{d}{|G|}+1$ is
an upper bound on the genus of $\widetilde{C}_Z$.

In \cite{Debarre}, it is explained that genus $2$ curves $C$ in $|\hat{L}|$
correspond uniquely, up to translation, to isogenies $a_*:J\widetilde{C}\to \hat{A}$ with kernels of order $d$. The number of isomorphism classes of such isogenies is equal to the number of isomorphism classes of isogenies $a^*:A\to J\widetilde{C}$ with kernel of order $d$, via taking the dual.
Thus we have proven the 
``if and only if'' statement of (a) in the case when $C_Z$ has geometric genus $2$.

In \cite{Zahariuc}, the author considers the Severi variety parametrizing curves in a general polarized abelian surface. More specifically, this Severi variety parametrizes unramified maps from smooth curves 
 of a fixed genus $g>2$ 
to the abelian surface where their images have N\'eron--Severi class equal to that of the polarization.
The irreducible components of this Severi variety for the abelian surface $\hat{A}$ are characterized as follows \cite[Theorem~1.1]{Zahariuc}:
if $(\hat{A},\hat{l})$ is a general polarized abelian surface, then 
each component corresponds to a distinct isogeny (up to isomorphism)
$\theta:A'\to \hat{A}$ of degree dividing $d$ that $\phi_L$ factors through. The maps from smooth curves to $\hat{A}$ that are parametrized by this component are those that factor sharply through $\theta$.
By \Cref{curve.inc} and \Cref{rmk.curveinc}, $-\phi_L=a_*\circ a^*$, and so 
the component of the Severi variety that parametrizes $a$ is determined by~$G_Z$.
\end{proof}

\begin{rmk}\label[remark]{nonnodal_Klein}
In \cite[Lemma~2.1]{KLC}, the authors characterize genus $2$ curves 
with a non-nodal singularity in terms of the Weierstrass points of the curve's normalization.
When $C$ is as in \Cref{genus2etc} and $d=4$, 
Knutsen and Lelli-Chiesa's criteria implies 
$C_Z$ is a genus $2$ curve with a non-nodal singularity 
(a triple point in this case)
if and only if $G_Z$ is the unique Klein four-group contained in $K(L)$.
\end{rmk}

\section{Properties of $\Psi(Z)$ for symmetric $Z$}\label{sec.compute}

In this section, we assume that the 
line bundle $L$ defining $\Psi$
is symmetric. 

\subsection{The map $(\star)$}\label{sec.star}
In this section, we show how to use eigenvalues of $\iota^*$ acting on vector spaces associated to $Z$ to analyze $\Psi$, introducing quadratic forms as a tool for these computations, and we prove \Cref{eigen}. 

\begin{lemma}[{cf.~(3.1.3)~{\cite{Gulbrandsen}}}]
Let $Z\in \Hilb^d_A$ and $C_Z$ be the supporting curve of $\Psi(Z)$. 
For any $x\in \hat{A}$, $x\in C_Z$ if and only if the following natural map is not an isomorphism:
\[
H^0(A, L\otimes P_x)\xrightarrow{(\star)} H^0(A,\O_Z\otimes L\otimes P_x).
\]
\end{lemma} 

\begin{proof}
By cohomology and base change, the fiber of $\Psi(Z)$ over $x$ is 
\[\Psi(Z)\otimes k(x)\cong H^1(A,I_Z\otimes L\otimes P_x),\]
and thus
\[C_Z=\{x\in\hat{A}~|~H^1(A,I_Z\otimes L\otimes P_x)\neq0\}.\]
First, consider the short exact sequence
\[0\to I_Z\to\O_C\to\O_Z\to 0,\]
then tensor it with $L\otimes P_x$, and take global sections cohomology.
Since $L\otimes P_x$ is ample and hence $H^0(A,L\otimes P_x)\neq 0$, by Mumford's Index Theorem
\cite[\S16]{Mumford}, we have $H^1(A,L\otimes P_x)=0$. Thus we obtain the following 
long exact sequence:
\begin{equation}\label{stareqn}
  0\to H^0(A,I_Z\otimes L\otimes P_x)\to H^0(A, L\otimes P_x)\xrightarrow{(\star)} H^0(A,\O_Z\otimes L\otimes P_x)\to H^1(A,I_Z\otimes L\otimes P_x)\to 0.
\end{equation}
Since $H^0(A, L\otimes P_x)$ and $H^0(A,\O_Z\otimes L\otimes P_x)$
have the same dimension $d$, we may conclude.
\end{proof}

\begin{rmk}
The map $(\star)$ restricts functions in the linear system of $L\otimes P_x$ to $Z$.
\end{rmk}

If $Z$ is symmetric and $x\in\hat{A}[2]$, then $\iota^*$ 
acts as an automorphism on \eqref{stareqn}. 
If the eigenvalues of this action on the domain and codomain of $(\star)$ differ, then $x\in C_Z$. We now discuss how to compute these eigenvalues.

The eigenvalues of the action of $\iota^*$ on $H^0(A,L\otimes P_x)$ are 
given by \Cref{dimcount}, and depend on the properties of $L\otimes P_x$.

Now we consider the action on $H^0(A,\O_Z\otimes L\otimes P_x)$.
If $Z$ is symmetric, its support consists of points that are not fixed by $\iota$ and points that are fixed by $\iota$, i.e. $2$-torsion points.
For each $z\in\Supp(Z)$ where $-z\neq z$, 
the action of $\iota^*$ exchanges $L\otimes P_x\otimes \O_Z|_{z}$
and $L\otimes P_x\otimes \O_Z|_{-z}$; this contributes eigenvalues $1$ and $-1$ each with multiplicity $l=\length(\O_Z|_{z})$.
For each $a\in \Supp(Z)$ where $-a= a$, 
the eigenvalues are computed using the quadratic form $q$ determined by $L\otimes P_x$ (see \Cref{sym.sec}). 
Let $M:=L\otimes P_x\otimes\O_Z|_a$ and consider the associated graded module 
\begin{equation}\label{asgr}
\gr M\cong M/\fm_aM\oplus \fm_aM/\fm_a^2M\oplus\cdots\oplus\fm_a^tM/\fm_a^{t+1}M.
\end{equation}
The action of $\iota^*$ on $\fm_a^iM/\fm_a^{i+1}M$ is multiplication by $(-1)^iq_{L\otimes P_x}(a)$ (cf.~\cite[\S13.2]{Polishchukbook}), 
so the portion of $Z$ supported on $a$ contributes
eigenvalue $q(a)$ with multiplicity $\sum_{i=0}^\infty\dim_{\C}\fm_a^{2i}M/\fm_a^{2i+1}M$ and eigenvalue $-q(a)$ with multiplicity $\sum_{i=0}^\infty\dim_{\C}\fm_a^{2i+1}M/\fm_a^{2i+2}M$.

We can think of the determinant of $\iota^*$ acting on 
$H^0(A,\O_Z\otimes L\otimes P_x)$ as an extension of the quadratic form $q$ to symmetric $d$-points.
The following useful observation is then a generalization of the fact that 
quadratic forms are all trivial at $e_A$. 

\begin{prop}\label[prop]{detZ}
Let $Z\in \Hilb^d_A$ be a symmetric $d$-point where $\Sigma Z=e_A$. 
The determinant of $\iota^*$ acting on $H^0(A,\O_Z\otimes L\otimes P_x)$ is independent of $x\in \hat{A}[2]$.
\end{prop}

\begin{proof}
By our discussion of the eigenvalues of the action of $\iota^*$ on 
$H^0(A,\O_Z\otimes L\otimes P_x)$ above, varying~$x$ will not affect the eigenvalues contributed by points in the support of $Z$ that are not fixed by $\iota$. 

We now assume that the support of $Z$ consists of points $a_1,\ldots a_k\in A[2]$. Let $l_i$ be the length of $\O_Z|_{a_i}$. The condition $\Sigma Z=e_A$ implies $l_1a_1+\cdots +l_ka_k=e_A$ in the group law. 

For each $a_i$, the dimensions of the components of the decomposition
\eqref{asgr} are determined by the structure of $\O_Z|_{a_i}$ and are independent of the choice of $x$.
Let $n_i:=\sum_{i=0}^\infty\dim_{\C}\fm_a^{2i}M/\fm_a^{2i+1}M$, 
$m_i:=\sum_{i=0}^\infty\dim_{\C}\fm_a^{2i+1}M/\fm_a^{2i+2}M$, and 
let $q$ be the quadratic form associated to $L$. 
Noting that $l_i=m_i+n_i$, we have
\begin{align*}
\det(\iota^*\mid H^0(A,\O_Z\otimes L))
&=q(a_1)^{n_1}(-q(a_1)^{m_1})\cdots 
q(a_k)^{n_k}(-q(a_k)^{m_k})\\
&=(-1)^{m_1+\cdots+m_k}q(a_1)^{l_1}\cdots q(a_k)^{l_k}.
\end{align*}
Let $x\in \hat{A}[2]$ and let $q_x$ be the quadratic form determined by $L\otimes P_x$. Then, similarly, 
\[
\det(\iota^*\mid H^0(A,\O_Z\otimes L\otimes P_x))
=(-1)^{m_1+\cdots+m_k}q_x(a_1)^{l_1}\cdots q_x(a_k)^{l_k}.
\]
Using the fact that $e_2$ is bilinear and for
any $a\in A[2]$ we have
$q_x(a)=q(a)\cdot e_2(a,x)$ 
\Cref{qfaction}, we obtain:
\begin{align*}
\det(\iota^*\mid H^0(A,\O_Z\otimes L\otimes P_x))
&=(-1)^{m_1+\cdots+m_k}q(a_1)^{l_1}e_2(a_1,x)^{l_1}\cdots q(a_k)^{l_k}
e_2(a_k,x)^{l_k}\\
&=(-1)^{m_1+\cdots+m_k}(q(a_1)^{l_1}\cdots q(a_k)^{l_k})
e_2(l_1a_1+\cdots l_ka_k,x)\\
&=(-1)^{m_1+\cdots+m_k}q(a_1)^{l_1}\cdots q(a_k)^{l_k}
e_2(e,x)\\
&=(-1)^{m_1+\cdots+m_k}q(a_1)^{l_1}\cdots q(a_k)^{l_k}.\qedhere
\end{align*}
\end{proof}

By \Cref{albhilb}, for any
$Z\in K_{d-1}A$, $C_Z$ is a member of the linear system~$|\hat{L}|$. 
We may now complete the proof of \Cref{eigen}.

\begin{thm}[{{\Cref{eigen}}}]\label{eigen.body}
Let $Z\in K_{d-1}A$ be symmetric.
\begin{enumerate}[(a)]
\item The eigensystem $|\hat{L}|^{\pm}$ containing $C_Z$ is
specified by the determinant of the action of $\iota$ on~$Z$.
\item For each $2$-torsion point $x\in\hat{A}[2]$, there is a lower bound on the dimension of 
$\Psi(Z)\otimes k(x)$ that we can compute
with
a quadratic form. 
\end{enumerate}  
\end{thm}

\begin{proof}
We will obtain this result by analyzing the determinant of the action of $\iota^*$ on the domain and codomain of $(\star)$.

(b) We established above that the fiber of $\Psi(Z)$ at $x$ is $H^1(A,I_Z\otimes L\otimes P_x)$. 

If the eigenvalues of $\iota^*$ acting on the domain and codomain of \eqref{stareqn} differ, then \eqref{stareqn} fails to be an isomorphism and the difference of eigenvalues determines a lower bound on 
$h^1(A,I_Z\otimes L\otimes P_x)$:
  \begin{align*}
    h^1(A,I_Z\otimes L\otimes P_x)&\geq|h^0(A,L\otimes P_x)^+-h^0(A,\O_Z\otimes L\otimes P_x)^+|\\
    &=|h^0(A,L\otimes P_x)^--h^0(A,\O_Z\otimes L\otimes P_x)^-|.
  \end{align*}
The eigenvalues are computed using \Cref{dimcount} and the quadratic form determined by $L\otimes P_x$, 
so the result is proven.

(a) We can use the idea from the proof of part (b). 
If the determinants of $\iota^*$ on the domain and codomain of \eqref{stareqn} are different, then their eigenvalues are different.

First, we observe that by \Cref{dimcount}, the determinant of the domain of \eqref{stareqn} depends on properties of $L\otimes P_x$: 
\begin{align*}
d\equiv 1(3)\mod 4 \quad&\Rightarrow\quad
\det(\iota^*\mid H^0(L\otimes P_x))=
\begin{cases}
\text{$1(-1)$ if $L\otimes P_x$ is even,} \\
\text{$-1(1)$ if $L\otimes P_x$ is odd.}
\end{cases}
\\
d\equiv 0(2)\mod 4 \quad&\Rightarrow\quad
\det(\iota^*\mid H^0(L\otimes P_x))=
\begin{cases}
\text{$-1(1)$ if $L\otimes P_x$ has sts},\\
\text{$1(-1)$ if $L\otimes P_x$ has no sts}.
\end{cases}
\end{align*}

By \Cref{detZ}, $\det(\iota^* |\, H^0(A,\O_Z\otimes L\otimes P_x))$ 
is independent of $x$. 
Thus, if $d$ is odd, $C_Z$ must pass through either $6$ or $10$ points of $\hat{A}[2]$, corresponding to the even or odd line bundles. 
Similarly, if $d$ is even, $C_Z$ must pass through either $4$ or $12$ points of $\hat{A}[2]$, corresponding to line bundles $L\otimes P_x$ that have or do not have sts. 

If we compose the isomorphism $\Psi:K_{d-1}A \to K_{\hat{A}}(0,\hat{l},-1)$ with the map to the supporting curve $C_Z$, we have a surjection
$K_{d-1}A\to |\hat{L}|$, which maps the symmetric $Z$ onto the eigensystems $|\hat{L}^{\pm}|$ (see \Cref{sec.moduli}).

Thus this determinantal condition has partitioned the curves in $|\hat{L}^{\pm}|$ in terms of passing through complementary sets of $2$-torsion points. The base loci of $\hat{L}^{\pm}$ also a partition $\hat{A}[2]$. 
These partitions must coincide: 
if they did not, we find a contradiction in the dimensions of the linear system
because the sub-system of curves
passing through a point outside the base locus 
has codimension~$1$ in a linear system.
\end{proof}

\begin{rmk}
The proof of (b) also shows we may 
 determine the eigenvalues of $\iota^*$ acting on a subspace $\Psi(Z)\otimes k(x)$ of dimension equal to the lower bound.
\end{rmk}

\begin{cor}\label[cor]{oddodd}
Let $d$ be even. 
\begin{enumerate}[(a)]
\item If $L$ is symmetric, then $\hat{L}$ has sts. If $L$ has sts, then $\hat{L}$ is odd.
\item Although the double-dual $L^{\wedge\wedge}$ always has the same N\'eron--Severi class as $L$, we have $L\simeq L^{\wedge\wedge}$ if and only if $L$ is odd.
\end{enumerate}
\end{cor}

\begin{proof}
(a) 
In the proof of \Cref{eigen.body}(a), we see that for any choice of $L$ that is symmetric, the base locus of $|\hat{L}|^+$ (resp. $|\hat{L}|^-$) consists of $x\in \hat{A}[2]$ so that $L\otimes P_x$ has (resp. does not have) sts.
As $x\in \hat{A}[2]$ varies, the set of all $L\otimes P_x$ is the set of all symmetric line bundles. 
By \cite[6.9.5]{BirkenhakeLange},
the symmetric line bundles on $\hat{A}$ with sts are a principal homogeneous space over $\hat{A}[2]/(K(\hat{L})\cap\hat{A}[2])$, hence there are $4$ of them. Thus the eigenspaces $|\hat{L}|^{\pm}$ have $4$ and $12$ base points and by \Cref{dimcount}, $\hat{L}$ must have sts.

Also by \Cref{dimcount}, we can detect whether $\hat{L}$ is even or odd
depending on which base locus contains the identity element $e$: $\hat{L}$ is odd if $e$ is in the base locus of the eigensystem that has $4$ base points, and even if it is in the base locus with $12$ base points. Thus, if $L$ has sts, $\hat{L}$ is odd (and if $L$ has no sts, $\hat{L}$ is even).

(b) Odd line bundles are unique, and by (a), the dual of an odd line bundle is odd. By \Cref{PhiL_lemma}(c) (also \cite{BLdual}), the N\'eron--Severi class of $L$ is the same as $L^{\wedge\wedge}$. Applying \Cref{PhiL_lemma}(c) twice, we have:
\[
(L\otimes P_x)^{\wedge\wedge}\simeq 
(\hat{L}\otimes\phi_{\hat{L}}(x))^{\wedge}
\simeq 
L^{\wedge\wedge}\otimes P_{dx}.
\]
If we choose $L$ to be odd, the right-hand side is 
$L\otimes P_{dx}$, which is only isomorphic to $L\otimes P_x$ when $x$ is trivial.
\end{proof}

\subsection{Singularities of $C_Z$}
In this section we discuss further results that allow us to make deductions about singularities of supporting curves $C_Z$. By work of Lange--Sernesi, genus $2$ curves in abelian surfaces can only contain ordinary singularities 
\cite[Prop~2.2]{LangeSernesi}. More recently, Knutsen--Lelli-Chiesa have shown that in fact those singularities must be nodal unless $4\nmid d$ \cite[Thm.~1]{KLC}. When $4|d$, the non-nodal singularities must be ordinary triple, $4$-tuple, or $6$-tuple points \cite[Rmk.~1]{KLC}. Since we will 
be focused on the behavior of curves at $2$-torsion points in the surface,
we note the following: 

\begin{lemma}\label[lemma]{nocusp}
A symmetric curve in an abelian surface cannot have a cusp singularity at a $2$-torsion point. 
\end{lemma}

\begin{proof}
  If a symmetric curve $C$ in an abelian surface has a cusp singularity at a $2$-torsion point $x$, then 
$\iota$ must act identically on the the tangent space of $C$ at $x$. However, this is not possible because the involution $\iota$ acts on the tangent space of $A$ at $x$ by central inversion. 
\end{proof}

An important tool for us will be the dimension of the fiber of  $\Psi(Z)$ at a point, which gives a lower bound for the multiplicity of $C_Z$ at that point.
We begin by reminding the reader the definition of multiplicity.

\begin{defn}
  Let $(R,\fm)$ be a Noetherian local ring of dimension $d$. The (Hilbert--Samuel) function $\operatorname{length}_R(R/\fm^n)$ is eventually a polynomial $P(n)$ of degree $d$ \cite[00L8]{stacks-project}. The multiplicity of $R$, denoted $e(R)$, is the leading coefficient of this polynomial times $d!$. If $X$ is a scheme and $x\in X$ is a point, we define the multiplicity of $X$ at $x$, denoted $e_x(X)$, to be the multiplicity of the local ring $\O_{X,x}$.
\end{defn}

Curves in an abelian surface are necessarily Gorenstein and thus Cohen--Macaulay. 
As a result, we may apply the following lemma, which we expect is well-known to experts, to the curves $C_Z$.

\begin{lemma}\label[lemma]{sally}
  Let $\fF$ be a rank 1 torsion free sheaf on a Cohen--Macaulay curve $C$ and consider the fiber $\fF_x/\fm_x\fF_x$ at a point $x\in C$. Then $\dim_\C\fF_x/\fm_x\fF_x\leq e_x(C)$.
\end{lemma}

\begin{proof}
  First, suppose $\fF$ is an ideal sheaf of $\O_X$. Then $\dim_\C\fF_x/\fm_x\fF_x$ is the minimal number of generators of the ideal $\fF_x$. Since $\O_{X,x}$ is Gorenstein and thus Cohen--Macaulay, we have that this is bounded above by $e_x$ (see, for example, \cite[Thm.~3.1.1]{Sally}).

  More generally, writing $\fF\cong I_Z\otimes\O_C(D)$ for $I_Z\subseteq\O_X$ and $D$ an effective Cartier divisor, we have that $\dim_\C\fF_x/\fm_x\fF_x=\dim_\C (I_Z)x/\fm_x(I_Z)_x$ and we are done.
\end{proof}

\subsection{Eigenvalues of $\Psi(Z)$}\label{sec.eval}
The results in this section show how the eigenvalues of $\Psi(Z)$ can be used to recover information about $C_Z$, particularly when $C_Z$ is hyperelliptic.

\begin{lemma}\label[lemma]{invol}
Let $X$ be a smooth curve with a nontrivial involution $j:X\to X$ and 
$D=\sum_{P\in X} n_PP$ a Weil divisor fixed by $j$. 
For each $x\in X$ 
fixed by $j$,
the eigenvalue of $j^*$ acting on the fiber of $\O(D)$ at $x$ is $(-1)^{n_x}$. 
\end{lemma}

\begin{proof}
Since $j(D)=D$, we have a canonical linearization $j^*\O_X(D)\to \O_X(D)$ that maps each section $f$ to $f\circ j$. We may choose a uniformizer $t$ of $\O_{X,x}$ so that $j^*t=-t$. The local generator of the stalk $\O_X(D)_x$ is $t^{-n_x}$, and hence the eigenvalue of the action of $j^*$ on the fiber is given by $(-1)^{-n_x}=(-1)^{n_x}$.
\end{proof}

\begin{rmk}\label[remark]{eval.switch}
In the above lemma, the linearization on $\O(D)$ coming from the choice of $D$ determines the eigenvalues. Given a different linearization, the eigenvalues will either be the same or all be multiplied by $-1$.
In general a  line bundle fixed by $j$, independent of the choice of linearization, determines a \emph{partition} of fixed points by the eigenvalues of their fibers.
\end{rmk}

A smooth, projective curve $C$ of genus $g\geq 2$ is hyperelliptic if it admits a degree $2$ map to $\mathbb{P}^1$. By Riemann--Hurwitz, there are $2g+2$ ramification points where $g$ is the genus of $C$.
Since we are working over $\C$, a field of charactertic not equal to $2$, this equips $C$ with a nontrivial involution $j:C\to C$ called the hyperelliptic involution. We call a divisor of the form $H=P+j(P)$ for $P\in C$ a hyperplane divisor, as it is the preimage of a hyperplane (a point) in $\mathbb{P}^1$. 
Any two hyperplane divisors are linearly equivalent.

\begin{lemma}\label[lemma]{sym.divisor}
Let $C$ be any smooth, projective hyperelliptic curve with hyperelliptic involution $j:C\to C$. For any line bundle $\L$ on $C$ so that $j^*\L\simeq \L$, there is a $D$ such that $j(D)=D$ and $\L\simeq \O(D)$.
\end{lemma}

\begin{proof}
Given any effective divisor $D$ on a hyperelliptic curve $C$, 
we may write it uniquely as $D=E+F$, 
where $E$ and $F$ are effective and $E$ is a maximal sum of divisors linearly equivalent to $H$. 
If the degree of $D$ is equal to the genus $g$ of $C$, 
then the Abel map $A:C^{(g)}\to J^g_C$ is surjective and its fiber
over $|D|$ is $n$-dimensional where $E\sim nH$. This fiber
consists of all divisors of the form 
$F+E'$ for $E'\sim nH$ (see  \cite[\S1]{Kass}). 

Thus for any degree $g$ line bundle $\L$ on $C$, $\L\simeq \O(D)$ for some effective $D=E+F$ as above.
We note that $j(D)=j(E)+j(F)=E+j(F)$, and $j(D)$ is linearly equivalent to $D$ if and only if $F=j(F)$. Thus, if $\L$ is symmetric, $F=j(F)$, and thus $D=j(D)$.

Let $P\in C$ be a fixed point of $j$.
For any degree $d$ line bundle $\L$ on $C$, $\L\simeq \O(D+(d-g)P)$ where $D$ is effective and degree $g$.
If $j^*\L\simeq \L$, then $j^*\O(D)\simeq \O(D)$, and we may conclude.
\end{proof}

\begin{lemma}\label[lemma]{lem.unique}
Let $\L$ be a symmetric line bundle 
on a smooth, hyperelliptic curve $C$ where the degree $d\geq0$ of $\L$ has the same parity as the genus $g$ of $C$. 
The partition of the $j$-fixed points of $C$ 
given by $\L$ (see \Cref{eval.switch})
uniquely determines 
the isomorphism class of $\L$.
\end{lemma}

\begin{proof}
Let $\L$ be a symmetric line bundle on $C$ of degree $d$ and let  $P_1,\ldots,P_{2g+2}$ be the ramification points of $C$. 
As in the proof \Cref{sym.divisor}, we may write 
$\L\simeq\O(D)$ for a symmetric divisor $D$ of the form $D=E+F+(d-g)[P]$.
Since the divisor $F$ contains no hyperplane divisors, it is a sum of distinct ramification points of the hyperelliptic curve, so, after possibly renaming the ramification points, we have $F=P_{1}+\cdots+P_{m}$ for some $m\leq g$. 

The coefficent in $D$ of 
any point $P_i$ in the support of $F$ must be odd
and the coefficient of any ramification point not in the support of $F$ must be even.
By \Cref{invol} and \Cref{eval.switch}, the partition of ramification points by eigenvalue given by $\L$ consists of 
$\{P_{1},\ldots,P_{m}\}\sqcup \{P_{m+1},\ldots,P_{2g+2}\}$.

Now, suppose we have another symmetric line bundle $\L'$ on $C$  
that has the same degree as $\L$ and has the same partition of ramification points by eigenvalue.
As above, we may write $\L'\simeq\O(D')$ for $D'=E'+F'+(d-g)[P]$ 
Since the partitions of $\L$ and $\L'$ are the same, 
$F'=P_1+\cdots+P_m$ or $F'=P_{m+1}+\cdots+P_{2g+2}$. 
However, since $m=\deg F\leq g$ and $\deg F'\leq g$, we must have $F'=F$ and thus $D\sim D'$.
\end{proof}

\begin{rmk}
  \label[remark]{rmk.summary}
To summarize, given a smooth hyperelliptic curve $C$ with hyperelliptic involution $j$ with ramification points $P_1,\dots,P_{2g+2}$ and symmetric line bundle $\L$ of degree $d\equiv g\mod 2$, we may write $\L\simeq\O(D)$ for a symmetric divisor $D=nH+F$ with
\begin{itemize}
  \item $\{P_1,\dots,P_m\}\sqcup\{P_{m+1},\dots,P_{2g+2}\}$ the partition of $j$-fixed points given by $\L$ and $m\leq g$,
  \item $F=P_1+\dots+P_m$,
  \item $2n+m=d$.
\end{itemize}

We may also write $\L\cong\O(D')$ where $D'\sim(n+m-(g+1))H+F'$ with $F'=P_{m+1}+\cdots+P_{2g+2}$. Although the linearizations given by $D$ and $D'$ determine different eigenvalues on the fibers, we have
$D\sim D'$. Indeed, by \cite[\S5.2.2]{Dolgachev}, for any fixed ramification point $Q$:
$$P_1+\cdots+P_m-mQ\sim P_{j+1}+\cdots+P_{2g+2}-(2g+2-m)Q.$$
\end{rmk}

\subsection{$\Psi(Z)$ on singular curves}

Finally, we consider rank $1$ torsion free sheaves on singular curves, which will help us analyze examples in later sections.

Let $\fF$ be a rank $1$ torsion-free sheaf on a nodal curve $C$. If $x$ is a node, then $\fF_x$ is either isomorphic to $\O_{C,x}$, in which case the fiber is $1$-dimensional, or to the maximal ideal $\fm_x$ of $\O_{C,x}$, in which case the fiber is $2$-dimensional. Given that $\fm_x\cong f_*\O_{\widetilde{C},x}$ where $f:\widetilde{C}\to C$ is the normalization of $C$, this leads us to the following.

\begin{lemma}[Prop.~10.1~\cite{OdaSeshadri}]\label[lemma]{pushnode}
Let $C$ be a projective curve with only nodal singularities. If $\fF$ is a rank $1$ torsion-free sheaf that has a $2$-dimensional fiber at each singular point, then $\fF$ is the pushforward of a line bundle on the normalization of $C$.
\end{lemma}

For integral curves with other singularities, we can consider the discussion presented in \cite[\S4.2.4]{Dolgachev}. Given $\fF$ a rank $1$ torsion-free sheaf, consider the sheaf of algebras $\sEnd(\fF)$. This embeds into $\sEnd(\fF_\eta)$, where $\eta$ is the generic point of $C$, which is isomorphic to $K(C)$, which tells us that $\sEnd(\fF)$ is a sheaf of coherent $\O_C$-algebras. Setting $C'=\Spec(\sEnd(\fF))$, we have a finite morphism $\pi:C'\to C$. The normalization $\widetilde{C}\to C$ factors through $\pi$, and so $C'$ is often called a \emph{partial normalization} of $C$.

$C'$ comes equipped with a sheaf $\fF'$ such that $\fF\simeq\pi_*\fF'$. In fact, $\fF'=\sHom_{\O_C}(\pi_*\O_{C'},\fF)$ viewed as an $\O_{C'}$-module. Furthermore, when $C'$ is Gorenstein, $\fF'$ is a line bundle on $C'$. This is guaranteed when $C$ has only $A_n$ singularities (such as nodes or cusps), but 
in general $C'$ may not be Gorenstein and $\fF'$ may not be a line bundle.

In the $(1,4)$-polarized case, we encounter curves with an ordinary triple point, which is a $D_4$ singularity; a partial normalization that is not Gorenstein exists. 
However, we are able to adapt the proof of \cite[Prop.~8.I.2]{Seshadri} to show that certain sheaves are nevertheless pushforwards of line bundles.

\begin{lemma}
  \label[lemma]{lem.triple}
Let $C$ be a projective plane curve with one ordinary triple point 
singularity. If $\fF$ is a rank $1$ torsion-free sheaf that has a $3$-dimensional fiber at the singular point, then, up to tensoring with a line bundle on $C$, $\fF$ is the pushforward of a line bundle on the normalization of $C$.
\end{lemma}
\begin{proof}
  Up to tensoring with a line bundle, rank one torsion-free sheaves on plane curves are entirely determined by their isomorphism classes at the singular points of the curve, so we can work locally. Let $x$ be the ordinary triple point and $f\colon\widetilde{C}\to C$ be the normalization map. Consider the normalization map $\O_{C,x}\to (f_*\O_{\widetilde{C}})_x$. We may assume that $\O_{C,x}\subseteq \fF_x\subseteq (f_*\O_{\widetilde{C}})_x$ and consider the natural map $\fF_x\to (f_*\O_{\widetilde{C}})_x/\O_{C,x}\cong\C^2$. The image of this map is a subspace of $\C^2$ of dimension one less than $\dim_\C \fF_x/\fm_x\fF_x$. Thus, if $\fF_x/\fm_x\fF_x$ is $3$-dimensional, $\fF_x\twoheadrightarrow (f_*\O_{\widetilde{C}})_x/\O_{C,x}$ and therefore $\fF_x\cong (f_*\O_{\widetilde{C}})_x$. Now consider $\widetilde{\fF}:=f^*\fF/\text{torsion}$. This is a line bundle on $\widetilde{C}$ and there is a short exact sequence
  \[0\to\fF\to f_*\widetilde{\fF}\to T\to 0\]
where 
$T$ is torsion and its support is contained in $\{x\}$.
Since $\fF_x\cong (f_*\O_{\widetilde{C}})_x$, we have that $T_x=0$ and thus $T=0$. Therefore, $\fF\cong f_*\widetilde{\fF}$.
\end{proof}

We finish with the following lemma, a small generalization of \cite[Thm.~1]{Northcott}, which we expect is well-known to experts. In what follows, given a rank one torsion-free sheaf $\fF$ on a curve $C$, we write $\fF^\vee$ for $\sHom(\fF,\O_C)$.

\begin{lemma}
  \label[lemma]{lem.conductor}
  Let $C$ be a plane curve, smooth except for a single ordinary singularity of multiplicity $m$ at $p\in C$. Let $\nu\colon\widetilde{C}\to C$ be the normalization map. If $D=\sum_{\nu(q)=p}q$, then for $k\geq0$,
  \[\nu_*\O_{\widetilde{C}}(kD)=(\I_p^{m+k-1})^\vee.\]
\end{lemma}

\begin{proof}
  Since $\nu_*\O_{\widetilde{C}_Z}(kD)$ is rank one torsion-free and hence reflexive, we prove the dual statement $\nu_*\O_{\widetilde{C}}(kD)^\vee=(\I_p^{m+k-1})$. When $k=0$, this is simply \cite[Thm.~1]{Northcott}, as the first neighborhood ring agrees with the normalization for an ordinary plane curve singularity. For $k\geq1$, we claim that $I_p^{m+k-1}\subseteq(\nu_*\O_{\widetilde{C}}(kD))^\vee$. To see this, note that if we take $f\in\I_p^{m+k-1}$ and write $f=\sum g_ih_i$ for $g_i\in\I_p^{m-1}$ 
and $h_i\in\I_p^k$ then
  \[g_ih_i\nu_*\O_{\widetilde{C}}(kD))\subseteq g_i\nu_*\O_{\widetilde{C}}\subseteq\O_C\]
  and therefore $f\nu_*\O_{\widetilde{C}}(kD)\subseteq\O_C$.
  The result then follows because the length of $\O_C/\I_p^{m+k-1}$ is $\frac{1}{2}m(m-1)$, the same as the length of $\O_C/\nu_*\O_{\widetilde{C}}(kD)^\vee$.
\end{proof}

\section{Singularities of arbitrarily high multiplicity}\label{sec.nonnodal}

In \cite{KLC}, Knutsen and Lelli-Chiesa prove that, in a general $(d_1,d_2)$-polarized abelian surface, there are curves of geometric genus $2$ in $|L|$ 
with non-nodal singularities if and only if $4|d_2$. 
These non-nodal singularities may be ordinary triple, $4$-tuple, or $6$-tuple points.

When we use our methods to examine the singularities of curves with geometric genus greater than $2$, the situation is less constrained. We exhibit two families of curves with singularities that have arbitrarily high multiplicity. 

Our first set of examples occurs in cases where $d$ is a multiple of $4$ and achieves singularities of arbitrarily high odd multiplicity.

\begin{example}\label{ex.oddmult}
In the case where $d=4$ and 
$Z=(e,k_1,k_2,k_3)$ where 
$k_i$ are the three nontrivial $2$-torsion points in $K(L)$, $C_Z$ has an ordinary triple point at $e_{\hat{A}}$ as shown in \cite{KLC}.

We furthermore consider the subschemes $Z_n$ with ideal sheaf $I_{Z_n}=\prod_{p\in Z}I_p^n$, where $I_p$ is the ideal sheaf of the point $p$. 
The length of these subschemes is four times the $n$-th triangle number, hence $d=2n(n+1)$.
As in \Cref{eigen.body}(b), we calculate a lower bound on the 
dimension of the fiber of $\Psi(\O_{Z_n})$ at $e_{\hat{A}}$.

We set the convention that $L$ 
has a symmetric theta structure and is odd. The associated quadratic form acts trivally on all entries of $Z$ (see~\Cref{tab:qform_labels}). Then, we may deduce from our discussion in \Cref{sec.compute} that 
the eigenvalues of the action of $\iota^*$ on the restrictions $\O_{Z_n}|_p$, $p\in Z$, are $1$ and $-1$ with multiplicity given by the number of monomials $x^ay^b$ with $a+b<n$ and $a+b$ even or odd, respectively. 
When $n$ is odd, we thus have:
\begin{equation}
h^0(A,\O_{Z_n}\otimes L)^+=(n+1)^2.
\end{equation}
By \Cref{dimcount}, 
\begin{equation}
h^0(A,L)^+=n(n+1)-1,
\end{equation}
and thus the dimension of the fiber of $\Psi(Z)$ at $e$ is bounded below by $(n+1)^2-(n(n+1)-1)=n+2$. 

By \Cref{sally}, the supporting curves of $\Psi(Z_n)$ thus have singularities at $e$ of multiplicity of at least $n+2$. 
The determinant of the action of $\iota^*$ on $H^0(A,\O_{Z_n}\otimes L)$ is $1$
and thus, by the proof of \Cref{eigen.body}(a), the supporting curve of $\Psi(Z_n)$ is in the eigensystem of $|L|$ that has the points $e$, $k_1$, $k_2$ and $k_3$ as its base locus, so the singularity at $e$ must have odd multiplicity by \Cref{prop.eigencurve.mult}.

In each of these examples the group $G_{Z_n}$ of \Cref{structure_thm_intro} is $\langle e, k_1,k_2,k_3\rangle$, so the genus of the curve $C_Z$ is $2$ but for $n>1$, the genus of $C_{Z_n}$ is strictly greater than $2$. 
\end{example} 

We next show a family of examples that has singularities of arbitrarily high even multiplicity where $d$ is even. 

\begin{example}\label{highmult.even}
We consider the $6$-point $Z=(e,k_1,w_1,k_1+w_1,w_2,k_1+w_2)$, and, in a similar strategy to \Cref{ex.oddmult}, we generalize it to 
subschemes $Z_n$ with ideal sheaf $I_{Z_n}=\prod_{p\in Z}I_p^n$, 
where $I_p$ is the ideal sheaf of the point $p$, which have length $3n(n+1)$.

We observe that $q_{y_1}$ is trivial at all points in $Z$ (see \Cref{tab:qform_labels}). When $n$ is odd, we 
use \Cref{eigen.body}(b) and \Cref{dimcount} to 
calculate a lower bound on the 
dimension of the fiber of $\Psi(\O_{Z_n})$ at $y_1$:
\begin{equation}
h^0(A,\O_{Z_n}\otimes L\otimes P_{y_1})-h^0(A,L)^+=
\tfrac{3}{2}(n+1)^2-\tfrac{3}{2}n(n+1)=\tfrac32 (n+1).
\end{equation}
Thus we see that these curves attain arbitrarily high multiplicities at $y_1$. Because $G_{Z_n}=\langle e,k_1\rangle$
and $d\geq 6$ in each of these examples, by 
\Cref{structure_thm_intro} the genera of these curves are greater than~$2$.

We restrict further to the case where $n$ is odd and $n\equiv 1\mod 4$. 
To determine the parity of the multiplicity we translate $Z_n$ so that we can use the results of \Cref{eigen.body}.
In this case $\frac{n(n+1)}{2}$ is odd, so 
$\sum Z_n=\frac{n(n+1)}{2}k_1=k_1$. Let $k_1'\in A[4]$ so that $2k_1'=k_1$. 
Consider $Z_n+k_1'$, the translation of $Z_n$ by $k_1'$. By \Cref{lem.trten},
the supporting curve of  $\Psi(Z_n+k_1')$ is the translation of $C_{Z_n}$ by $\phi_L(k_1')=y_1$, so the singularity we analyzed above has been translated to~$e$. The action of $\iota$ on the 
translation $Z+k_1'$ exchanges three pairs of points and so the determinant of the action of 
$\iota^*$ on $H^0(A,\O_{Z+k_1'}\otimes L)$ is $-1$.
The 
determinant of the action of $\iota^*$ on 
$H^0(A,\O_{{Z_n}+k_1'}\otimes L)$ is $-1$ as well. 
Referring again to the ideas in \Cref{eigen.body}(a), $e$ is not a base point of the eigen-space of $|L|$ containing $C_{Z_n}$, so the multiplicity of $C_{Z_n}$ at $e$ must in fact be even. 
\end{example} 

We have now proven \Cref{highmult} because 
\Cref{ex.oddmult} and \Cref{highmult.even} allow us to 
produce curves of geometric genus greater than $2$ with singularities of arbitrarily high multiplicities of either parity.

\begin{thm}[\Cref{highmult}]
Let $m\in\N$.
There are infinitely many $d\in\N$ so that every 
$(1,d)$-polarized abelian surface of Picard rank $1$ contains a curve 
with even (resp.\ odd) multiplicity at least $m$ in its polarizing linear system.
\end{thm}

Finally, the lowest value of $d$ for which we are able to prove there is a curve of genus greater than $2$ that has a singularity of multiplicity greater than $2$ is $d=5$.

\begin{example}
Let $d=5$. In this case, the quadratic forms corresponding to the symmetric line bundles $L\otimes P_x$ do not behave as in \Cref{tab:qform_labels}.
There are ten even forms and six odd forms. Each of the even forms evaluates trivially at ten points in ${A}[2]$ and non-trivially at the other six points
(see for instance \cite{HonigsMcDonald}). We select an even form $q_p$ corresponding to a point $p\in \hat{A}[2]$. Let $Z$ be a length-$5$ subscheme supported at five distinct points $a_1,\ldots,a_5\in A[2]$ where 
$q_p(a_i)=-1$ for all $1 \leq i\leq 5$, and thus $h^0(\O_Z\otimes L\otimes P_p)^-=5$. 
By \Cref{dimcount}, $h^0(A,L\otimes P_p)^-=2$, thus $C_Z$ has a singularity at $p$ of multiplicity at least $3$.
\end{example}

\section{The $(1,4)$-polarized case}\label{sec.140ex}

\subsection{Analysis of isolated fixed points}
We now turn our attention to the case where $d=4$. In particular, we choose $L$ to be symmetric and odd. 
Our goal is to analyze the isolated fixed points of the action of $\iota$ on $K_{\hat{A}}(0,\hat{l},-1)$ by using the isomorphism $\Psi:K_{3}A\to K_{\hat{A}}(0,\hat{l},-1)$.

Since $\iota$ is symplectic and both $K_{3}A$ and $K_{\hat{A}}(0,\hat{l},-1)$ are hyperk\"ahler $6$-folds of Kummer type, 
the fixed locus of $\iota$ acting on each of these spaces 
consists of one fourfold and 
$140$ isolated fixed points \cite{KMO}.
In $K_3A$, the fixed fourfold is the closure of 
$\{(x,-x,y,-y)| x,y\in A\}$. 
The isolated fixed points are precisely the 
$4$-points $(a_1,a_2,a_3,a_4)$, where 
$a_1,a_2,a_3,a_4\in A[2]$ are distinct points such that $a_1+a_2+a_3+a_4=e_A$. We call this set of isolated fixed points $S\subseteq K_3A$.
To analyze the isolated fixed points of $K_{\hat{A}}(0,\hat{l},-1)$, we analyze $\Psi(S)$.

By Lemmas \ref{albhilb} and \ref{PhiL_lemma}, 
the space $K_{\hat{A}}(0,\hat{l},-1)$ parametrizes rank $1$ torsion-free sheaves $\fF\simeq I_Z\otimes \O_C(D)$ 
that have Euler characteristic $-1$, are supported on curves in $|\hat{L}|$, and have $\Sigma D-\Sigma Z=e$.

Each $Z\in S$ is a translation of the Klein group, and so the group $T_{Z}=\{x\in A\mid t_x(Z)=Z\}$ is a Klein group. However, the 
groups $G_Z=T_Z\cap K(L)$ 
are different, partitioning the $140$ points of $S$ into three sets, which we call Type 1, where $G_Z$ is trivial, Type 2, where $G_Z$ is a cyclic group of order $2$, and Type 3, where $G_Z$ is the Klein group. 

Each $Z\in S$ is listed in the appendix along with the eigenvalues of $\Psi(Z)\otimes k(x)$ that are be detected using the methods of \Cref{sec.star}
for each $x\in \hat{A}[2]$. The values of $Z$ are split up in the appendix according to their orbits under the isometry group of $A[2]$ with respect to the degenerate polarization pairing $e_L$. These orbits respect the curve types we have identified.

The following discussion proves \Cref{140}. In particular, we describe the singularities of the supporting curves $C_Z$ as well as the divisor class correpsonding to each $\Psi(Z)$.

\subsection*{Type 1} There are $64$ points in $S$ of this type. 
These $64$ points are also distinguished because they are the values of $Z\in S$ where 
$C_Z$ is in the eigensystem $|L|^+$, which has a unique curve. 
This curve is, furthemore,
 the unique (up to translation) smooth hyperelliptic curve on a general $(1,4)$-polarized abelian surface. The involution $\iota$ on $\hat{A}$ restricts to the hyperelliptic involution of this curve; the $12$ base points of $|L|^+$ are the ramification points.
This curve was analyzed in \cite{BorowkaOrtega} and shown to be 
invariant under a group of translations isomorphic to the Klein group.
Since each $Z$ in the Type $1$ case
is invariant under a Klein group of translations whose image under $\phi_L$ is $\langle x_1,x_2\rangle$, by 
\Cref{rmk.trten}(b) 
$t_{x_1}$ and $t_{x_2}$ are automorphisms of $C_Z$.

Since $C_Z$ is a smooth curve of genus $5$, each sheaf $\Psi(Z)$ is a line bundle of degree $3$. The kernel of the map $a_*: \Pic^0(C)\to A$ 
defined in \Cref{curve.inc} consists of degree $0$ line bundles $\O(D)$ on $C$ where $\sum D=0$.
This kernel 
 is a connected, $3$-dimensional abelian variety (cf.~
\cite[Lemma~2.6]{BSgenus4curves}, \cite[Prop.~6.6]{Frei_Honigs}).
The $64$ line bundles $\Psi(Z)$ are a torsor over the $2$-torsion points of this abelian $3$-fold.
These $64$ line bundles are furthermore cosets of 
the order $4$ subgroup group
$\{P_k\mid k\in K(L)\cap A[2]\}\leq \ker(a_*)$, which are precisely the degree $0$ line bundles on $\hat{A}$ whose restrictions to $C$ are in the kernel of $a_*$.
By \Cref{lem.trten}, the cosets correspond to 
$4$-points $Z\in S$ 
that are an orbit under translation by 
elements of $K(L)\cap A[2]=\{e,k_1,k_2,k_3\}$. 
The $Z\in S$ of Type $1$ are split into 
Tables~\ref{type.one.one} and \ref{type.one.two} in the appendix, which are two orbits under the isometry action, which does not respect these other cosets.

Finally, we can use the results of \Cref{sec.eval} (see \Cref{rmk.summary}) to 
match each $Z$ with the divisor class corresponding to $\Psi(Z)$. 
We consider one example each from 
Tables~\ref{type.one.one} and \ref{type.one.two}.

\begin{example}
If $Z=(e,w_1,w_2,w_3)$, then $\Psi(Z)\cong\O_{C_Z}(D)$ where, for instance
  \[
  D=y_1+y_2+y_3\sim \sum_{i=1}^3\sum_{j=1}^3(x_i+y_j)-3H.
  \]
In either example, $\sum D=e$ in the group law of $\hat{A}$, and so by \Cref{detPhi} and \Cref{detPhidual} $\det(\Phi(\Psi(Z)))\cong\det(\Phi(\hat{L}))$ as expected.
\end{example}

The examples in \Cref{type.one.one} correspond to all possible partitions of the ramification points into 
$3$ and $9$ points that each sum to $e$ in the group law. In \Cref{type.one.two}, on the other hand, the ramification points are partitioned into $5$ and $7$ points that sum to $e$.

\begin{example}
  If $Z=(w_1,k_1,w_2,w_3+k_1)$, then $\Psi(Z)\cong\O_{C_Z}(D)$ where, for instance
\begin{align*}D&=(x_1+y_1)+(x_2+y_1)+\sum_{i=1}^3(x_3+y_i)-H.\\
&\sim y_1+y_2+y_3+(x_1+y_2)+(x_1+y_3)+(x_2+y_2)+(x_2+y_3)-2H.
\end{align*}
In either example, $\sum D=e$ in the group law of $\hat{A}$, and so by \Cref{detPhi} and \Cref{detPhidual} $\det(\Phi(\Psi(Z)))\cong\det(\Phi(\hat{L}))$ as expected.
\end{example}

\subsection*{Type 2} There are $72$ elements $Z\in S$ where $G_Z$ is a group of order $2$, which we call Type 2. 
By \Cref{genus2etc}, each of the supporting curves $C_Z$ has geometric genus $3$. Furthermore, each 
$C_Z$ is contained in the net $|L|^-\simeq \P^2$, which has base points $e,x_1,x_2,x_3$.

These $4$-points have three orbits under the isometry group, shown in Tables~\ref{type.two.one} and \ref{type.two.two}.
In these tables, we see that 
each $\Psi(Z)$ has two fibers of dimension at least $2$ at points of $\hat{A}[2]$ outside of the base locus of $|L|^-$, which are labelled as $n_1$ and $n_2$ in the tables. 
By \Cref{sally}, these points are singularities of $C_Z$;
because these cannot be cusp singularities (\Cref{nocusp}), by genus bounds they must both be nodal singularities, and so the fiber dimensions at these nodal singularities are equal to $2$. 

Among the Type $2$ points $Z$, the supporting curves $C_Z$ have $18$ different pairs of nodes.
The nodes range over all possible choices of two points in $\hat{A}[2]$ outside the base locus of $|L|^-$ that sum to a point in the base locus. Each curve is invariant under translation by that point; the translation exchanges the two nodes. 

There are four sheaves $\Psi(Z)$ whose supporting curve $C_Z$ has a given pair of nodes. 
Two pairs of these sheaves $\Psi(Z)$ are certainly supported on the same curve $C_Z$ because the $Z$'s are related by translation, but the pairs are not related by translation and it is not clear to the authors that the supporting curves $C_Z$ are all the same.

We now turn to the identification of $\Psi(Z)$. Let $\nu\colon\widetilde{C}_Z\to C_Z$ be the normalization map. The involution $\iota$ on each $C_Z$ extends to an involution $j$ on its normalization. This normalization is genus $3$ and $j$ has $8$ fixed points, making it a hyperelliptic involution. Because each $\Psi(Z)$ has a two dimensional fiber at both nodes of $C_Z$, \Cref{pushnode} tells us that $\Psi(Z)=\nu_*(\fF_Z)$ where $\fF_Z$ is a line bundle on $\widetilde{C}_Z$. The eigenvalues of $\iota$ acting on the fiber of $\Psi(Z)$ at each node correspond to the eigenvalues of $j$ acting on the fiber of $\fF_Z$ at the two points lying over the node in $\widetilde{C}_Z$. The results of \Cref{sec.eval} then allow us to produce divisors $D$ such that $\fF_Z\cong\O_{\widetilde{C}_Z}(D)$. Note that because $\tilde{C}_Z$ has genus $3$, and $\fF_Z$ and $\nu_*(\fF_Z)$ have equal Euler characteristics, 
$\fF_Z$ must have degree $1$.

To standardize notation, let $p_i$ and $q_i$ be the preimages of the node $n_i$ under the normalization map. Note that $\nu_*\O_{\widetilde{C}_Z}\cong(\I_{n_1,n_2})^\vee$. 
We show examples from each of the three orbits shown in Tables~\ref{type.two.one} and \ref{type.two.two}; these orbits are differentiated by the eigenvalues at the nodal points.

\begin{example}
  Let $Z$ be the $4$-point $(e,w_1,k_1,w_1+k_1)$. Based on the partition of the 8 ramification points indicated in \Cref{type.two.one}, we can write $\fF_Z\cong\O_{\widetilde{C}_Z}(x_1)$ in which case
  \[\Psi(Z)\cong\nu_*\O_{\widetilde{C}_Z}(x_1)\cong \O_{C_Z}(x_1)\otimes (\I_{n_1,n_2})^\vee.\]
  Alternatively, we could write $\fF_Z\cong\O_{\widetilde{C}_Z}(D)$ for
  $D=e+x_2+x_3+p_1+q_1+p_2+q_2-3H$
  in which case
  \[\Psi(Z)\cong\nu_*\O_{\widetilde{C}_Z}(D)\cong\O_{C_Z}(e+x_2+x_3)\otimes(\I_{n_1,n_2}^2)^\vee\]
  Note that $x_1+n_1+n_2=e=e+x_2+x_3+3n_1+3n_2$ in the group law of $\hat{A}$, and so by \Cref{detPhi} and \Cref{detPhidual}, $\det(\Phi(\Psi(Z)))\cong\det(\Phi(\hat{L}))$ as expected.
\end{example}

\begin{example}
  Let $Z$ be the $4$-point $(w_1,k_1,k_2,w_1+k_3)$. Based on the partition of the $8$ ramification points shown in \Cref{type.two.two}, we can write $\fF_Z\cong\O(D)$ where
  \[D=e+x_2+x_3-H\]
  in which case
  \[\Psi(Z)\cong\nu_*\O_{\widetilde{C}_Z}(D')\cong\O_{C_Z}(e+x_2+x_3-H)\otimes(\I_{n_1,n_2})^\vee.\]
  Alternatively, we 
may choose
$D=x_1+p_1+q_1+p_2+q_2-2H$,
  in which case
  \[\Psi(Z)\cong\nu_*\O_{\widetilde{C}_Z}(D)\cong\O(x_1-2H)\otimes(\I_{n_1,n_2}^2)^\vee.\]
  Note that $e+x_2+x_3+n_1+n_2=e=x_1+3n_1+3n_2$ in the group law of $\hat{A}$, and so by \Cref{detPhi} and \Cref{detPhidual} $\det(\Phi(\Psi(Z)))\cong\det(\Phi(\hat{L}))$ as expected.
\end{example}

\begin{example}
  Let $Z$ be the $4$ point $(w_1,w_1+k_1,w_2,w_2+k_1)$. Based on the partition of the $8$ ramification points given in \Cref{type.two.two}, we can write $\fF_Z=\O_{\widetilde{C}_Z}(D)$ where
  \[D=x_3+p_2+q_2-H\]
  in which case
  \[\Psi(Z)\cong\nu_*\O_{\widetilde{C}_Z}(D')\cong\O_{C_Z}(x_3-H)\otimes(\I_{n_1}\I_{n_2}^2)^\vee.\]
  Alternatively, we could write 
$D=e+x_1+x_2+p_1+q_1-2H$,
  in which case
  \[\Psi(Z)\cong\nu_*\O_{\widetilde{C}_Z}(D)\cong\O_{C_Z}(e+x_1+x_2-2H)\otimes(\I_{n_1}^2\I_{n_2})^\vee.\]
  Note that $x_3+n_1+3n_2=e=e+x_1+x_2+3n_1+n_2$ in the group law of $\hat{A}$, and so by \Cref{detPhi} and \Cref{detPhidual}, we have $\det(\Phi(\Psi(Z)))\cong\det(\Phi(\hat{L}))$ as expected.
\end{example}

\subsection*{Type 3} Finally, we have $4$ examples $Z\in S$ where the group $G_Z$ is the Klein four group, split into two orbits of the isometry group as shown in \Cref{type.three}.
The curves $C_Z$ are in $|\hat{L}|^-$. 
By \Cref{structure_thm_intro}, for each $Z$ of Type 3, the curve $C_Z$ has geometric genus $2$. 
As we see in the table,
each sheaf $\Psi(Z)$ has rank at least $3$ at one of the points in the base locus of $|\hat{L}|^-$, so $C_Z$ has a singularity of multiplicity at least $3$ at that point (\Cref{sally}). By \cite[Prop.~2.2]{LangeSernesi}, the singularities of $C_Z$ must be ordinary, which tells us that $C_Z$ is smooth except for an ordinary triple point at that singularity.
This discussion partially recovers a result of \cite{KLC}
(cf.~\Cref{nonnodal_Klein}), which shows that $C_Z$ has an ordinary triple point if and only if $G_Z$ is a Klein group.
By \Cref{lem.trten}, since the four $Z$ of Type $3$ are translations of one another by $w_1$, $w_2$, $w_3$, the supporting curves $C_Z$ are translations of one another.

The involution $\iota$ on each $C_Z$ extends to an involution $j$ on the normalization $\widetilde{C}_Z$, which has $6$ fixed points and is genus $2$, so is hyperelliptic. 
As each $\Psi(Z)$ has fiber dimension $3$ at the singular point of $C_Z$, \Cref{lem.triple} tells us that $\Psi(Z)=\nu_*\fF_Z$ where $\nu\colon\widetilde{C}_Z\to C_Z$ is the normalization map and $\fF_Z$ is a line bundle. The eigenvalues of $\iota$ acting on the fiber of $\Psi(Z)$ at the triple point correspond to the eigenvalues of $j$ acting on the fiber of $\fF_Z$ at the three points in $\widetilde{C}_Z$ lying over the triple point of $C_Z$. The results of \Cref{sec.eval} then allow us to find $D$ such that $\fF_Z\cong\O_{\widetilde{C}_Z}(D)$. Note that because $\tilde{C}_Z$ has genus $2$, $\fF_Z$ must have degree $0$.

We find some divisors $D$ for two examples from \Cref{type.three}. 

\begin{example}
Let $Z=(e,k_1,k_2,k_3)$. 
Let $a$, $b$, and $c$ be the preimages of the triple point $e$ under the normalization map to $C_Z$.
The partition of the ramification points of $\widetilde{C}_Z$ 
given in \Cref{type.three}
consists of all the ramification points and the empty set, which 
lets us write $\fF_Z\cong\O_{\widetilde{C}_Z}$ in which case, by \Cref{lem.conductor}
  \[\Psi(Z)\cong\nu_*\O_{\widetilde{C}_Z}\cong(I_e^2)^\vee.\]
  Alternatively, we could write
  $\fF_Z\cong\O_{\widetilde{C}_Z}(D)$ where
  \[D=a+b+c+x_1+x_2+x_3-3H,\]
  in which case
  \[\Psi(Z)\cong\nu_*\O_{\widetilde{C}_Z}(D)\cong\O_{C_Z}(x_1+x_2+x_3-3H)\otimes(\I_e^3)^\vee.\]
  Note that $6e+x_1+x_2+x_3=e$ in the group law of $\hat{A}$ and so by \Cref{detPhi} and \Cref{detPhidual}, we have $\det(\Phi(\Psi(Z)))\cong\det(\Phi(\hat{L}))$ as expected.
\end{example}

\begin{example}
  Let $Z=(w_1,w_1+k_1,w_1+k_2,w_1+k_3)$. Let $a$, $b$, and $c$ be the preimages of the triple point $x_1$ under the normalization map to $C_Z$.
The partition of the ramification points of $\widetilde{C}_Z$ lets us write $\fF_Z\cong\O_{\widetilde{C}_Z}(D)$ where, we can take $D$ to be, for instance,
\[D=x_2+x_3-H, \]
  in which case
  \[\Psi(Z)\cong\nu_*\O_{\widetilde{C}_Z}(D)\cong \O_{C_Z}(x_2+x_3-H)\otimes(\I_{x_1}^2)^\vee.\]
  Alternatively, we could write
  $D=e+a+b+c-2H$,
  in which case
  \[\Psi(Z)\cong\nu_*\O_{\widetilde{C}_Z}(D)\cong \O_{C_Z}(e-2H)\otimes(\I_{x_1}^3)^\vee.\]
  Note that $e+6x_1=e=3x_1+x_2+x_3$ in the group law of $\hat{A}$, and so by \Cref{detPhi} and \Cref{detPhidual}, we have $\det(\Phi(\Psi(Z)))\cong\det(\Phi(\hat{L}))$ as expected.
\end{example}

\subsection{Questions and remarks}

Since all of the supporting curves $C_Z$ identified are hyperelliptic 
and there are only finitely many total hyperelliptic curves in $|\hat{L}|$, we might ask how many of these hyperelliptic curves are curves we have identified (up to translation).
It was shown in \cite[Table~1]{Bryan} that 
the linear system $|\hat{L}|$ contains $112$ hyperelliptic curves of (geometric) genus $2$, $456$ curves of genus $3$, $192$ curves of genus $4$ and $4$ curves of genus $5$. 

As discussed in \Cref{genus2etc}, distinct genus $2$ curves correspond to choices of order $4$ subgroups of $K(L)$. There are $7$ such subgroups -- one is the Klein group and the others are cyclic. In general, given  a curve in $|\hat{L}|$, its translation by each of the $16$ points in $K(\hat{L})$ is also in $|\hat{L}|$, so we have accounted for all $7\cdot 16=112$ total genus $2$ curves. 
In the set of examples we treat, we see only the curves where $G_Z$ the Klein group, which accounts for only $16$ of the total curves. 
Moreover, by \cite{KLC}, the curves where $G_Z$ is cyclic of order $4$ must have three nodal singularities. Examining such a case, if we 
let $k_1'\in K(L)$ be $4$-torsion and $2k_1'=k_1$, then 
$Z=(e,k_1',k_1,-k_1')$ has a group $G_Z$ that is cyclic of order $4$. Using our methods, we are able to detect one nodal singularity at a $2$-torsion point, but we are not able to detect the others. 

There is a unique (up to translation) hyperelliptic curve of genus $5$ in $|\hat{L}|$, which was among our examples. 
Since these curves are invariant under translation by the Klein group $\phi_L(T_Z)\leq K(\hat{L})$, the total number of translations by $K(\hat{L})$
is $4$, as counted in \cite{Bryan}. 

Among our $140$ examples, we have not detected any genus $4$ curves, but we have detected hyperelliptic genus $3$ curves with $2$ nodes. 
Examining the $18$ pairs of nodes, we see that pairs of them are translations by elements of $K(\hat{L})$ (note $x_1,x_2,x_3\in K(\hat{L})$ but $y_1,y_2,y_3\not\in K(\hat{L})$), so we have found at least $9$ distinct curves that are not translates of one another by elements of $K(\hat{L})$. Taking translations, we have found at least $8\cdot \frac{16}{2}=72$ genus $3$ hyperelliptic curves in $|\hat{L}|$ out of the total of  $456$.

In the paper \cite{HonigsMcDonald}, two of the authors carried out an analysis of isolated fixed points in $K_{\hat{A}}(0,\hat{l},-1)$ when $\hat{l}$ is a polarization of $(1,3)$-type. The methods of the current paper would allow us to give the precise correspondence discussed in 
\cite[\S6.3]{HonigsMcDonald}.
In this setting, the action of $\iota^*$ on $K_2A$ has $36$ isolated fixed points. 
Their images under $\Psi$ consist of $16$ line bundles on a smooth hyperelliptic curve and $20$ sheaves on curves with a single node at a $2$-torsion point, 
whose normalizations are also hyperelliptic. With all these examples in mind, we pose the following question: 

\begin{question} Let $\hat{l}$ be a polarization of $(1,d)$-type for $d\geq 2$. 
Are all isolated fixed points of $\iota^*$ acting on $K_{\hat{A}}(0,\hat{l},-1)$\ldots
\begin{enumerate}[(1)]
\item supported on hyperelliptic curves?
\item supported on curves whose singularities are at points in $\hat{A}[2]$?
\item sheaves where we can detect the fiber dimension at each point in 
$\hat{A}[2]$
via a difference of eigenvalues as in \Cref{eigen}(b)?
\end{enumerate}
\end{question}

In \cite{KLC}, Knutsen and Lelli-Chiesa give examples of genus $2$ curves
with $4$-tuple and $6$-tuple singularities that are not in a primitive polarization, unlike the example that has an ordinary triple point.
We further pose the following related question:

\begin{question}
Let $A$ be a general $(1,d)$-polarized abelian surface. 
Can a genus $2$ curve in $A$ have a $4$-tuple or $6$-tuple singularity?
\end{question}

\appendix
\section{Eigenvalues of fibers of $\Psi(Z)$}\label{appendix}

We are in the setting of \Cref{sec.140ex}: $A$ is a complex abelian surface with Picard rank $1$ with a $(1,4)$-polarization. The isolated fixed points $Z\in K_3A$ under the action of $\iota^*$ are listed below, along with the eigenvalues of $\Psi(Z)\otimes k(x)$ for each $x\in \hat{A}[2]$,  
detected using the methods of \Cref{sec.star}. \textsf{Sage} \cite{sage} code used to compute these eigenvalues is available on github \cite{HMM}.

We split up the points $Z$ into separate tables according to their orbits 
under the isometry group of $A[2]$ with respect to the degenerate polarization pairing $e_L$. These orbits respect the curve types we have identified in \Cref{sec.140ex}, which we also list. 

We use the notation defined in \Cref{bkgrd} for the points of $A[2]$ and $\hat{A}[2]$.
In order to manage the size of the tables, we have only listed the points 
$x\in \hat{A}[2]$ where we have detected eigenvalues of $\Psi(Z)\otimes k(x)$. 
In the case of the curves of Type $1$, we only detect eigenvalues at $y_1$, $y_2$, $y_3$ and $x_i+y_j$ for $1\leq i,j\leq 3$.
For each curve of Type $2$, we detect one eigenvalue at the points $e$, $x_1$, $x_2$ and $x_3$, as well as two eigenvalues each at two other points denoted $n_1$ and $n_2$, which are different in each case. We list the eigenvalues detected and then list the respective points that they are detected at.
Finally, for curves of Type $3$, we only detect eigenvalues at $e$, $x_1$, $x_2$ and $x_3$.

\bigskip

\begin{table}[hp!]
\centering
\caption{Type 1 curves, one orbit}
\label{type.one.one}
\begin{tabular}{|l | ccc | ccc | ccc | ccc|}

    \hline
    & \multicolumn{3}{c|}{$y$} & \multicolumn{3}{c|}{$x_1+y$} & \multicolumn{3}{c|}{$x_2+y$} & \multicolumn{3}{c|}{$x_3+y$}\\
    \hline

    $Z$ & $y_1$ & $y_2$ & $y_3$ & $y_1$ & $y_2$ & $y_3$ & $y_1$ & $y_2$ & $y_3$ & $y_1$ & $y_2$ & $y_3$\\
    \hline

    $(e,w_1,w_2,w_3)$ & $-1$ & $-1$ & $-1$ & $1$ & $1$ & $1$ & $1$ & $1$ & $1$ & $1$ & $1$ & $1$  \\

    $(e,w_1,w_2+k_1,w_3+k_1)$ & $-1$ & $1$ & $1$ & $1$ & $-1$ & $-1$ & $1$ & $1$ & $1$ & $1$ & $1$ & $1$ \\

    $(e,w_1,w_2+k_2,w_3+k_2)$ & $1$ & $-1$ & $1$ & $-1$ & $1$ & $-1$ & $1$ & $1$ & $1$ & $1$ & $1$ & $1$ \\

    $(e,w_1,w_2+k_3,w_3+k_3)$ & $1$ & $1$ & $-1$ & $-1$ & $-1$ & $1$ & $1$ & $1$ & $1$ & $1$ & $1$ & $1$ \\

    $(e,w_1+k_1,w_2,w_3+k_1)$ & $-1$ & $1$ & $1$ & $1$ & $1$ & $1$ & $1$ & $-1$ & $-1$ & $1$ & $1$ & $1$ \\

    $(e,w_1+k_1,w_3,w_2+k_1)$ & $-1$ & $1$ & $1$ & $1$ & $1$ & $1$ & $1$ & $1$ & $1$ & $1$ & $-1$ & $-1$ \\

    $(e,w_1+k_1,w_2+k_2,w_3+k_3)$ & $1$ & $1$ & $1$ & $-1$ & $1$ & $1$ & $1$ & $-1$ & $1$ & $1$ & $1$ & $-1$ \\

    $(e,w_1+k_1,w_3+k_2,w_2+k_3)$ & $1$ & $1$ & $1$ & $-1$ & $1$ & $1$ & $1$ & $1$ & $-1$ & $1$ & $-1$ & $1$ \\

    $(e,w_2,w_1+k_2,w_3+k_2)$ & $1$ & $-1$ & $1$ & $1$ & $1$ & $1$ & $-1$ & $1$ & $-1$ & $1$ & $1$ & $1$ \\

    $(e,w_2,w_1+k_3,w_3+k_3)$ & $1$ & $1$ & $-1$ & $1$ & $1$ & $1$ & $-1$ & $-1$ & $1$ & $1$ & $1$ & $1$ \\

    $(e,w_3,w_1+k_2,w_2+k_2)$ & $1$ & $-1$ & $1$ & $1$ & $1$ & $1$ & $1$ & $1$ & $1$ & $-1$ & $1$ & $-1$ \\

    $(e,w_3,w_1+k_3,w_2+k_3)$ & $1$ & $1$ & $-1$ & $1$ & $1$ & $1$ & $1$ & $1$ & $1$ & $-1$ & $-1$ & $1$ \\

    $(e,w_2+k_1,w_1+k_2,w_3+k_3)$ & $1$ & $1$ & $1$ & $1$ & $-1$ & $1$ & $-1$ & $1$ & $1$ & $1$ & $1$ & $-1$ \\

    $(e,w_2+k_1,w_1+k_3,w_3+k_2)$ & $1$ & $1$ & $1$ & $1$ & $1$ & $-1$ & $-1$ & $1$ & $1$ & $1$ & $-1$ & $1$ \\

    $(e,w_3+k_1,w_1+k_2,w_2+k_3)$ & $1$ & $1$ & $1$ & $1$ & $-1$ & $1$ & $1$ & $1$ & $-1$ & $-1$ & $1$ & $1$ \\

    $(e,w_3+k_1,w_1+k_3,w_2+k_2)$ & $1$ & $1$ & $1$ & $1$ & $1$ & $-1$ & $1$ & $-1$ & $1$ & $-1$ & $1$ & $1$ \\
    \hline
\end{tabular}
\end{table}

\newpage

{\small
\begin{table}[hp!]
\caption{Type 1, one orbit}\label{type.one.two}
\begin{tabular}{|l | ccc | ccc | ccc | ccc|}

    \hline
    & \multicolumn{3}{c|}{$y$} & \multicolumn{3}{c|}{$x_1+y$} & \multicolumn{3}{c|}{$x_2+y$} & \multicolumn{3}{c|}{$x_3+y$}\\
    \hline

    $Z$ & $y_1$ & $y_2$ & $y_3$ & $y_1$ & $y_2$ & $y_3$ & $y_1$ & $y_2$ & $y_3$ & $y_1$ & $y_2$ & $y_3$\\
    \hline

    $(w_1,k_1,w_2,w_3+k_1)$ & $-1$ & $-1$ & $-1$ & $1$ & $-1$ & $-1$ & $1$ & $-1$ & $-1$ & $1$ & $1$ & $1$ \\

    $(w_1,k_1,w_3,w_2+k_1)$ & $-1$ & $-1$ & $-1$ & $1$ & $-1$ & $-1$ & $1$ & $1$ & $1$ & $1$ & $-1$ & $-1$ \\

    $(w_1,k_1,w_2+k_2,w_3+k_3)$ & $1$ & $-1$ & $-1$ & $-1$ & $-1$ & $-1$ & $1$ & $-1$ & $1$ & $1$ & $1$ & $-1$ \\

    $(w_1,k_1,w_3+k_2,w_2+k_3)$ & $1$ & $-1$ & $-1$ & $-1$ & $-1$ & $-1$ & $1$ & $1$ & $-1$ & $1$ & $-1$ & $1$ \\

    $(w_1,w_2,k_2,w_3+k_2)$ & $-1$ & $-1$ & $-1$ & $-1$ & $1$ & $-1$ & $-1$ & $1$ & $-1$ & $1$ & $1$ & $1$ \\

    $(w_1,w_2,k_3,w_3+k_3)$ & $-1$ & $-1$ & $-1$ & $-1$ & $-1$ & $1$ & $-1$ & $-1$ & $1$ & $1$ & $1$ & $1$ \\

    $(w_1,w_3,k_2,w_2+k_2)$ & $-1$ & $-1$ & $-1$ & $-1$ & $1$ & $-1$ & $1$ & $1$ & $1$ & $-1$ & $1$ & $-1$ \\

    $(w_1,w_3,k_3,w_2+k_3)$ & $-1$ & $-1$ & $-1$ & $-1$ & $-1$ & $1$ & $1$ & $1$ & $1$ & $-1$ & $-1$ & $1$ \\

    $(w_1,w_2+k_1,k_2,w_3+k_3)$ & $-1$ & $1$ & $-1$ & $-1$ & $-1$ & $-1$ & $-1$ & $1$ & $1$ & $1$ & $1$ & $-1$ \\

    $(w_1,w_2+k_1,k_3,w_3+k_2)$ & $-1$ & $-1$ & $1$ & $-1$ & $-1$ & $-1$ & $-1$ & $1$ & $1$ & $1$ & $-1$ & $1$ \\

    $(w_1,w_3+k_1,k_2,w_2+k_3)$ & $-1$ & $1$ & $-1$ & $-1$ & $-1$ & $-1$ & $1$ & $1$ & $-1$ & $-1$ & $1$ & $1$ \\

    $(w_1,w_3+k_1,k_3,w_2+k_2)$ & $-1$ & $-1$ & $1$ & $-1$ & $-1$ & $-1$ & $1$ & $-1$ & $1$ & $-1$ & $1$ & $1$ \\

    $(k_1,w_1+k_1,w_2,w_3)$ & $-1$ & $-1$ & $-1$ & $1$ & $1$ & $1$ & $1$ & $-1$ & $-1$ & $1$ & $-1$ & $-1$ \\

    $(k_1,w_1+k_1,w_2+k_1,w_3+k_1)$ & $-1$ & $1$ & $1$ & $1$ & $-1$ & $-1$ & $1$ & $-1$ & $-1$ & $1$ & $-1$ & $-1$ \\

    $(k_1,w_1+k_1,w_2+k_2,w_3+k_2)$ & $1$ & $-1$ & $1$ & $-1$ & $1$ & $-1$ & $1$ & $-1$ & $-1$ & $1$ & $-1$ & $-1$ \\

    $(k_1,w_1+k_1,w_2+k_3,w_3+k_3)$ & $1$ & $1$ & $-1$ & $-1$ & $-1$ & $1$ & $1$ & $-1$ & $-1$ & $1$ & $-1$ & $-1$ \\

    $(k_1,w_2,w_1+k_2,w_3+k_3)$ & $1$ & $-1$ & $-1$ & $1$ & $-1$ & $1$ & $-1$ & $-1$ & $-1$ & $1$ & $1$ & $-1$ \\

    $(k_1,w_2,w_1+k_3,w_3+k_2)$ & $1$ & $-1$ & $-1$ & $1$ & $1$ & $-1$ & $-1$ & $-1$ & $-1$ & $1$ & $-1$ & $1$ \\

    $(k_1,w_3,w_1+k_2,w_2+k_3)$ & $1$ & $-1$ & $-1$ & $1$ & $-1$ & $1$ & $1$ & $1$ & $-1$ & $-1$ & $-1$ & $-1$ \\

    $(k_1,w_3,w_1+k_3,w_2+k_2)$ & $1$ & $-1$ & $-1$ & $1$ & $1$ & $-1$ & $1$ & $-1$ & $1$ & $-1$ & $-1$ & $-1$ \\

    $(k_1,w_2+k_1,w_1+k_2,w_3+k_2)$ & $1$ & $-1$ & $1$ & $1$ & $-1$ & $-1$ & $-1$ & $1$ & $-1$ & $1$ & $-1$ & $-1$ \\

    $(k_1,w_2+k_1,w_1+k_3,w_3+k_3)$ & $1$ & $1$ & $-1$ & $1$ & $-1$ & $-1$ & $-1$ & $-1$ & $1$ & $1$ & $-1$ & $-1$ \\

    $(k_1,w_3+k_1,w_1+k_2,w_2+k_2)$ & $1$ & $-1$ & $1$ & $1$ & $-1$ & $-1$ & $1$ & $-1$ & $-1$ & $-1$ & $1$ & $-1$ \\

    $(k_1,w_3+k_1,w_1+k_3,w_2+k_3)$ & $1$ & $1$ & $-1$ & $1$ & $-1$ & $-1$ & $1$ & $-1$ & $-1$ & $-1$ & $-1$ & $1$ \\

    $(w_1+k_1,w_2,k_2,w_3+k_3)$ & $-1$ & $1$ & $-1$ & $-1$ & $1$ & $1$ & $-1$ & $-1$ & $-1$ & $1$ & $1$ & $-1$ \\

    $(w_1+k_1,w_2,k_3,w_3+k_2)$ & $-1$ & $-1$ & $1$ & $-1$ & $1$ & $1$ & $-1$ & $-1$ & $-1$ & $1$ & $-1$ & $1$ \\

    $(w_1+k_1,w_3,k_2,w_2+k_3)$ & $-1$ & $1$ & $-1$ & $-1$ & $1$ & $1$ & $1$ & $1$ & $-1$ & $-1$ & $-1$ & $-1$ \\

    $(w_1+k_1,w_3,k_3,w_2+k_2)$ & $-1$ & $-1$ & $1$ & $-1$ & $1$ & $1$ & $1$ & $-1$ & $1$ & $-1$ & $-1$ & $-1$ \\

    $(w_1+k_1,w_2+k_1,k_2,w_3+k_2)$ & $-1$ & $1$ & $1$ & $-1$ & $1$ & $-1$ & $-1$ & $1$ & $-1$ & $1$ & $-1$ & $-1$ \\

    $(w_1+k_1,w_2+k_1,k_3,w_3+k_3)$ & $-1$ & $1$ & $1$ & $-1$ & $-1$ & $1$ & $-1$ & $-1$ & $1$ & $1$ & $-1$ & $-1$ \\

    $(w_1+k_1,w_3+k_1,k_2,w_2+k_2)$ & $-1$ & $1$ & $1$ & $-1$ & $1$ & $-1$ & $1$ & $-1$ & $-1$ & $-1$ & $1$ & $-1$ \\

    $(w_1+k_1,w_3+k_1,k_3,w_2+k_3)$ & $-1$ & $1$ & $1$ & $-1$ & $-1$ & $1$ & $1$ & $-1$ & $-1$ & $-1$ & $-1$ & $1$ \\

    $(w_2,w_3,k_2,w_1+k_2)$ & $-1$ & $-1$ & $-1$ & $1$ & $1$ & $1$ & $-1$ & $1$ & $-1$ & $-1$ & $1$ & $-1$ \\

    $(w_2,w_3,k_3,w_1+k_3)$ & $-1$ & $-1$ & $-1$ & $1$ & $1$ & $1$ & $-1$ & $-1$ & $1$ & $-1$ & $-1$ & $1$ \\

    $(w_2,w_3+k_1,k_2,w_1+k_3)$ & $-1$ & $1$ & $-1$ & $1$ & $1$ & $-1$ & $-1$ & $-1$ & $-1$ & $-1$ & $1$ & $1$ \\

    $(w_2,w_3+k_1,w_1+k_2,k_3)$ & $-1$ & $-1$ & $1$ & $1$ & $-1$ & $1$ & $-1$ & $-1$ & $-1$ & $-1$ & $1$ & $1$ \\

    $(w_3,w_2+k_1,k_2,w_1+k_3)$ & $-1$ & $1$ & $-1$ & $1$ & $1$ & $-1$ & $-1$ & $1$ & $1$ & $-1$ & $-1$ & $-1$ \\

    $(w_3,w_2+k_1,w_1+k_2,k_3)$ & $-1$ & $-1$ & $1$ & $1$ & $-1$ & $1$ & $-1$ & $1$ & $1$ & $-1$ & $-1$ & $-1$ \\

    $(w_2+k_1,w_3+k_1,k_2,w_1+k_2)$ & $-1$ & $1$ & $1$ & $1$ & $-1$ & $-1$ & $-1$ & $1$ & $-1$ & $-1$ & $1$ & $-1$ \\

    $(w_2+k_1,w_3+k_1,k_3,w_1+k_3)$ & $-1$ & $1$ & $1$ & $1$ & $-1$ & $-1$ & $-1$ & $-1$ & $1$ & $-1$ & $-1$ & $1$ \\

    $(k_2,w_1+k_2,w_2+k_2,w_3+k_2)$ & $1$ & $-1$ & $1$ & $-1$ & $1$ & $-1$ & $-1$ & $1$ & $-1$ & $-1$ & $1$ & $-1$ \\

    $(k_2,w_1+k_2,w_2+k_3,w_3+k_3)$ & $1$ & $1$ & $-1$ & $-1$ & $-1$ & $1$ & $-1$ & $1$ & $-1$ & $-1$ & $1$ & $-1$ \\

    $(k_2,w_1+k_3,w_2+k_2,w_3+k_3)$ & $1$ & $1$ & $-1$ & $-1$ & $1$ & $-1$ & $-1$ & $-1$ & $1$ & $-1$ & $1$ & $-1$ \\

    $(k_2,w_1+k_3,w_3+k_2,w_2+k_3)$ & $1$ & $1$ & $-1$ & $-1$ & $1$ & $-1$ & $-1$ & $1$ & $-1$ & $-1$ & $-1$ & $1$ \\

    $(w_1+k_2,k_3,w_2+k_2,w_3+k_3)$ & $1$ & $-1$ & $1$ & $-1$ & $-1$ & $1$ & $-1$ & $-1$ & $1$ & $-1$ & $1$ & $-1$ \\

    $(w_1+k_2,k_3,w_3+k_2,w_2+k_3)$ & $1$ & $-1$ & $1$ & $-1$ & $-1$ & $1$ & $-1$ & $1$ & $-1$ & $-1$ & $-1$ & $1$ \\

    $(k_3,w_1+k_3,w_2+k_2,w_3+k_2)$ & $1$ & $-1$ & $1$ & $-1$ & $1$ & $-1$ & $-1$ & $-1$ & $1$ & $-1$ & $-1$ & $1$ \\

    $(k_3,w_1+k_3,w_2+k_3,w_3+k_3)$ & $1$ & $1$ & $-1$ & $-1$ & $-1$ & $1$ & $-1$ & $-1$ & $1$ & $-1$ & $-1$ & $1$ \\
    \hline
\end{tabular}
\end{table}
}

\begin{table}
\caption{Type 2, two orbits}\label{type.two.one}
\begin{tabular}{| l | cccc || cc | ll |}

    \hline
    $Z$ & $e$ & $x_1$ & $x_2$ & $x_3$ & $n_1$ & $n_2$ & $n_1$ & $n_2$\\
    \hline

    $(e,w_1,k_1,w_1+k_1)$ & $1$ & $-1$ & $1$ & $1$ & $(1,1)$ & $(1,1)$ & $x_2+y_1$, & $x_3+y_1$\\

    $(e,w_1,k_2,w_1+k_2)$ & $1$ & $-1$ & $1$ & $1$ & $(1,1)$ & $(1,1)$ & $x_2+y_2$, & $x_3+y_2$\\

    $(e,w_1,k_3,w_1+k_3)$ & $1$ & $-1$ & $1$ & $1$ & $(1,1)$ & $(1,1)$ & $x_2+y_3$, & $x_3+y_3$\\

    $(e,k_1,w_2,w_2+k_1)$ & $1$ & $1$ & $-1$ & $1$ & $(1,1)$ & $(1,1)$ & $x_1+y_1$, & $x_3+y_1$\\

    $(e,k_1,w_3,w_3+k_1)$ & $1$ & $1$ & $1$ & $-1$ & $(1,1)$ & $(1,1)$ & $x_1+y_1$, & $x_2+y_1$\\

    $(e,k_1,w_1+k_2,w_1+k_3)$ & $1$ & $-1$ & $1$ & $1$ & $(1,1)$ & $(1,1)$ & $y_1$, & $x_1+y_1$\\

    $(e,k_1,w_2+k_2,w_2+k_3)$ & $1$ & $1$ & $-1$ & $1$ & $(1,1)$ & $(1,1)$ & $y_1$, & $x_2+y_1$\\

    $(e,k_1,w_3+k_2,w_3+k_3)$ & $1$ & $1$ & $1$ & $-1$ & $(1,1)$ & $(1,1)$ & $y_1$, & $x_3+y_1$\\

    $(e,w_1+k_1,k_2,w_1+k_3)$ & $1$ & $-1$ & $1$ & $1$ & $(1,1)$ & $(1,1)$ & $y_2$, & $x_1+y_2$\\

    $(e,w_1+k_1,w_1+k_2,k_3)$ & $1$ & $-1$ & $1$ & $1$ & $(1,1)$ & $(1,1)$ & $y_3$, & $x_1+y_3$\\

    $(e,w_2,k_2,w_2+k_2)$ & $1$ & $1$ & $-1$ & $1$ & $(1,1)$ & $(1,1)$ & $x_1+y_2$, & $x_3+y_2$\\

    $(e,w_2,k_3,w_2+k_3)$ & $1$ & $1$ & $-1$ & $1$ & $(1,1)$ & $(1,1)$ & $x_1+y_3$, & $x_3+y_3$\\

    $(e,w_3,k_2,w_3+k_2)$ & $1$ & $1$ & $1$ & $-1$ & $(1,1)$ & $(1,1)$ & $x_1+y_2$, & $x_2+y_2$\\

    $(e,w_3,k_3,w_3+k_3)$ & $1$ & $1$ & $1$ & $-1$ & $(1,1)$ & $(1,1)$ & $x_1+y_3$, & $x_2+y_3$\\

    $(e,w_2+k_1,k_2,w_2+k_3)$ & $1$ & $1$ & $-1$ & $1$ & $(1,1)$ & $(1,1)$ & $y_2$, & $x_2+y_2$\\

    $(e,w_2+k_1,k_3,w_2+k_2)$ & $1$ & $1$ & $-1$ & $1$ & $(1,1)$ & $(1,1)$ & $y_3$, & $x_2+y_3$\\

    $(e,w_3+k_1,k_2,w_3+k_3)$ & $1$ & $1$ & $1$ & $-1$ & $(1,1)$ & $(1,1)$ & $y_2$, & $x_3+y_2$\\

    $(e,w_3+k_1,k_3,w_3+k_2)$ & $1$ & $1$ & $1$ & $-1$ & $(1,1)$ & $(1,1)$ & $y_3$, & $x_3+y_3$\\
\hline

\end{tabular}

\bigskip

\bigskip 

\bigskip 

\begin{tabular}{| l | cccc || cc | ll |}

    \hline
    $Z$ & $e$ & $x_1$ & $x_2$ & $x_3$ & $n_1$ & $n_2$ & $n_1$ & $n_2$\\
    \hline

    $(w_1,k_1,k_2,w_1+k_3)$ & $1$ & $-1$ & $1$ & $1$ & $(-1,-1)$ & $(-1,-1)$ & $y_3$, & $x_1+y_3$\\

    $(w_1,k_1,w_1+k_2,k_3)$ & $1$ & $-1$ & $1$ & $1$ & $(-1,-1)$ & $(-1,-1)$ & $y_2$, & $x_1+y_2$\\

    $(w_1,w_1+k_1,k_2,k_3)$ & $1$ & $-1$ & $1$ & $1$ & $(-1,-1)$ & $(-1,-1)$ & $y_1$, & $x_1+y_1$\\

    $(k_1,w_1+k_1,k_2,w_1+k_2)$ & $1$ & $-1$ & $1$ & $1$ & $(-1,-1)$ & $(-1,-1)$ & $x_2+y_3$, & $x_3+y_3$\\

    $(k_1,w_1+k_1,k_3,w_1+k_3)$ & $1$ & $-1$ & $1$ & $1$ & $(-1,-1)$ & $(-1,-1)$ & $x_2+y_2$, & $x_3+y_2$\\

    $(k_1,w_2,k_2,w_2+k_3)$ & $1$ & $1$ & $-1$ & $1$ & $(-1,-1)$ & $(-1,-1)$ & $y_3$, & $x_2+y_3$\\

    $(k_1,w_2,k_3,w_2+k_2)$ & $1$ & $1$ & $-1$ & $1$ & $(-1,-1)$ & $(-1,-1)$ & $y_2$, & $x_2+y_2$\\

    $(k_1,w_3,k_2,w_3+k_3)$ & $1$ & $1$ & $1$ & $-1$ & $(-1,-1)$ & $(-1,-1)$ & $y_3$, & $x_3+y_3$\\

    $(k_1,w_3,k_3,w_3+k_2)$ & $1$ & $1$ & $1$ & $-1$ & $(-1,-1)$ & $(-1,-1)$ & $y_2$, & $x_3+y_2$\\

    $(k_1,w_2+k_1,k_2,w_2+k_2)$ & $1$ & $1$ & $-1$ & $1$ & $(-1,-1)$ & $(-1,-1)$ & $x_1+y_3$, & $x_3+y_3$\\

    $(k_1,w_2+k_1,k_3,w_2+k_3)$ & $1$ & $1$ & $-1$ & $1$ & $(-1,-1)$ & $(-1,-1)$ & $x_1+y_2$, & $x_3+y_2$\\

    $(k_1,w_3+k_1,k_2,w_3+k_2)$ & $1$ & $1$ & $1$ & $-1$ & $(-1,-1)$ & $(-1,-1)$ & $x_1+y_3$, & $x_2+y_3$\\

    $(k_1,w_3+k_1,k_3,w_3+k_3)$ & $1$ & $1$ & $1$ & $-1$ & $(-1,-1)$ & $(-1,-1)$ & $x_1+y_2$, & $x_2+y_2$\\

    $(w_2,w_2+k_1,k_2,k_3)$ & $1$ & $1$ & $-1$ & $1$ & $(-1,-1)$ & $(-1,-1)$ & $y_1$, & $x_2+y_1$\\

    $(w_3,w_3+k_1,k_2,k_3)$ & $1$ & $1$ & $1$ & $-1$ & $(-1,-1)$ & $(-1,-1)$ & $y_1$, & $x_3+y_1$\\

    $(k_2,w_1+k_2,k_3,w_1+k_3)$ & $1$ & $-1$ & $1$ & $1$ & $(-1,-1)$ & $(-1,-1)$ & $x_2+y_1$, & $x_3+y_1$\\

    $(k_2,k_3,w_2+k_2,w_2+k_3)$ & $1$ & $1$ & $-1$ & $1$ & $(-1,-1)$ & $(-1,-1)$ & $x_1+y_1$, & $x_3+y_1$\\

    $(k_2,k_3,w_3+k_2,w_3+k_3)$ & $1$ & $1$ & $1$ & $-1$ & $(-1,-1)$ & $(-1,-1)$ & $x_1+y_1$, & $x_2+y_1$\\
\hline

\end{tabular}
\end{table}

\begin{table}[hp!]
\caption{Type 2, one orbit}
\label{type.two.two}

\begin{tabular}{| l | cccc || cc | ll |}

    \hline
    $Z$ & $e$ & $x_1$ & $x_2$ & $x_3$ & $n_1$ & $n_2$ & $n_1$ & $n_2$\\
    \hline

    $(w_1,w_1+k_1,w_2,w_2+k_1)$ & $-1$ & $-1$ & $-1$ & $1$ & $(-1,-1)$ & $(1,1)$ & $y_1$, & $x_3+y_1$\\

    $(w_1,w_1+k_1,w_3,w_3+k_1)$ & $-1$ & $-1$ & $1$ & $-1$ & $(-1,-1)$ & $(1,1)$ & $y_1$, & $x_2+y_1$\\

    $(w_1,w_1+k_1,w_2+k_2,w_2+k_3)$ & $-1$ & $-1$ & $-1$ & $1$ & $(-1,-1)$ & $(1,1)$ & $x_1+y_1$, & $x_2+y_1$\\

    $(w_1,w_1+k_1,w_3+k_2,w_3+k_3)$ & $-1$ & $-1$ & $1$ & $-1$ & $(-1,-1)$ & $(1,1)$ & $x_1+y_1$, & $x_3+y_1$\\

    $(w_1,w_2,w_1+k_2,w_2+k_2)$ & $-1$ & $-1$ & $-1$ & $1$ & $(-1,-1)$ & $(1,1)$ & $y_2$, & $x_3+y_2$\\

    $(w_1,w_2,w_1+k_3,w_2+k_3)$ & $-1$ & $-1$ & $-1$ & $1$ & $(-1,-1)$ & $(1,1)$ & $y_3$, & $x_3+y_3$\\

    $(w_1,w_3,w_1+k_2,w_3+k_2)$ & $-1$ & $-1$ & $1$ & $-1$ & $(-1,-1)$ & $(1,1)$ & $y_2$, & $x_2+y_2$\\

    $(w_1,w_3,w_1+k_3,w_3+k_3)$ & $-1$ & $-1$ & $1$ & $-1$ & $(-1,-1)$ & $(1,1)$ & $y_3$, & $x_2+y_3$\\

    $(w_1,w_2+k_1,w_1+k_2,w_2+k_3)$ & $-1$ & $-1$ & $-1$ & $1$ & $(-1,-1)$ & $(1,1)$ & $x_1+y_2$, & $x_2+y_2$\\

    $(w_1,w_2+k_1,w_1+k_3,w_2+k_2)$ & $-1$ & $-1$ & $-1$ & $1$ & $(-1,-1)$ & $(1,1)$ & $x_1+y_3$, & $x_2+y_3$\\

    $(w_1,w_3+k_1,w_1+k_2,w_3+k_3)$ & $-1$ & $-1$ & $1$ & $-1$ & $(-1,-1)$ & $(1,1)$ & $x_1+y_2$, & $x_3+y_2$\\

    $(w_1,w_3+k_1,w_1+k_3,w_3+k_2)$ & $-1$ & $-1$ & $1$ & $-1$ & $(-1,-1)$ & $(1,1)$ & $x_1+y_3$, & $x_3+y_3$\\

    $(w_1+k_1,w_2,w_1+k_2,w_2+k_3)$ & $-1$ & $-1$ & $-1$ & $1$ & $(1,1)$ & $(-1,-1)$ & $x_1+y_3$, & $x_2+y_3$\\

    $(w_1+k_1,w_2,w_1+k_3,w_2+k_2)$ & $-1$ & $-1$ & $-1$ & $1$ & $(1,1)$ & $(-1,-1)$ & $x_1+y_2$, & $x_2+y_2$\\

    $(w_1+k_1,w_3,w_1+k_2,w_3+k_3)$ & $-1$ & $-1$ & $1$ & $-1$ & $(1,1)$ & $(-1,-1)$ & $x_1+y_3$, & $x_3+y_3$\\

    $(w_1+k_1,w_3,w_1+k_3,w_3+k_2)$ & $-1$ & $-1$ & $1$ & $-1$ & $(1,1)$ & $(-1,-1)$ & $x_1+y_2$, & $x_3+y_2$\\

    $(w_1+k_1,w_2+k_1,w_1+k_2,w_2+k_2)$ & $-1$ & $-1$ & $-1$ & $1$ & $(1,1)$ & $(-1,-1)$ & $y_3$, & $x_3+y_3$\\

    $(w_1+k_1,w_2+k_1,w_1+k_3,w_2+k_3)$ & $-1$ & $-1$ & $-1$ & $1$ & $(1,1)$ & $(-1,-1)$ & $y_2$, & $x_3+y_2$\\

    $(w_1+k_1,w_3+k_1,w_1+k_2,w_3+k_2)$ & $-1$ & $-1$ & $1$ & $-1$ & $(1,1)$ & $(-1,-1)$ & $y_3$, & $x_2+y_3$\\

    $(w_1+k_1,w_3+k_1,w_1+k_3,w_3+k_3)$ & $-1$ & $-1$ & $1$ & $-1$ & $(1,1)$ & $(-1,-1)$ & $y_2$, & $x_2+y_2$\\

    $(w_2,w_3,w_2+k_1,w_3+k_1)$ & $-1$ & $1$ & $-1$ & $-1$ & $(-1,-1)$ & $(1,1)$ & $y_1$, & $x_1+y_1$\\

    $(w_2,w_3,w_2+k_2,w_3+k_2)$ & $-1$ & $1$ & $-1$ & $-1$ & $(-1,-1)$ & $(1,1)$ & $y_2$, & $x_1+y_2$\\

    $(w_2,w_3,w_2+k_3,w_3+k_3)$ & $-1$ & $1$ & $-1$ & $-1$ & $(-1,-1)$ & $(1,1)$ & $y_3$, & $x_1+y_3$\\

    $(w_2,w_2+k_1,w_1+k_2,w_1+k_3)$ & $-1$ & $-1$ & $-1$ & $1$ & $(1,1)$ & $(-1,-1)$ & $x_1+y_1$, & $x_2+y_1$\\

    $(w_2,w_2+k_1,w_3+k_2,w_3+k_3)$ & $-1$ & $1$ & $-1$ & $-1$ & $(-1,-1)$ & $(1,1)$ & $x_2+y_1$, & $x_3+y_1$\\

    $(w_2,w_3+k_1,w_2+k_2,w_3+k_3)$ & $-1$ & $1$ & $-1$ & $-1$ & $(-1,-1)$ & $(1,1)$ & $x_2+y_2$, & $x_3+y_2$\\

    $(w_2,w_3+k_1,w_3+k_2,w_2+k_3)$ & $-1$ & $1$ & $-1$ & $-1$ & $(-1,-1)$ & $(1,1)$ & $x_2+y_3$, & $x_3+y_3$\\

    $(w_3,w_2+k_1,w_2+k_2,w_3+k_3)$ & $-1$ & $1$ & $-1$ & $-1$ & $(1,1)$ & $(-1,-1)$ & $x_2+y_3$, & $x_3+y_3$\\

    $(w_3,w_2+k_1,w_3+k_2,w_2+k_3)$ & $-1$ & $1$ & $-1$ & $-1$ & $(1,1)$ & $(-1,-1)$ & $x_2+y_2$, & $x_3+y_2$\\

    $(w_3,w_3+k_1,w_1+k_2,w_1+k_3)$ & $-1$ & $-1$ & $1$ & $-1$ & $(1,1)$ & $(-1,-1)$ & $x_1+y_1$, & $x_3+y_1$\\

    $(w_3,w_3+k_1,w_2+k_2,w_2+k_3)$ & $-1$ & $1$ & $-1$ & $-1$ & $(1,1)$ & $(-1,-1)$ & $x_2+y_1$, & $x_3+y_1$\\

    $(w_2+k_1,w_3+k_1,w_2+k_2,w_3+k_2)$ & $-1$ & $1$ & $-1$ & $-1$ & $(1,1)$ & $(-1,-1)$ & $y_3$, & $x_1+y_3$\\

    $(w_2+k_1,w_3+k_1,w_2+k_3,w_3+k_3)$ & $-1$ & $1$ & $-1$ & $-1$ & $(1,1)$ & $(-1,-1)$ & $y_2$, & $x_1+y_2$\\

    $(w_1+k_2,w_1+k_3,w_2+k_2,w_2+k_3)$ & $-1$ & $-1$ & $-1$ & $1$ & $(1,1)$ & $(-1,-1)$ & $y_1$, & $x_3+y_1$\\

    $(w_1+k_2,w_1+k_3,w_3+k_2,w_3+k_3)$ & $-1$ & $-1$ & $1$ & $-1$ & $(1,1)$ & $(-1,-1)$ & $y_1$, & $x_2+y_1$\\

    $(w_2+k_2,w_3+k_2,w_2+k_3,w_3+k_3)$ & $-1$ & $1$ & $-1$ & $-1$ & $(1,1)$ & $(-1,-1)$ & $y_1$, & $x_1+y_1$\\
\hline
\end{tabular}
\end{table}

\clearpage
\newpage

\begin{table}
\caption{Type 3, two orbits}\label{type.three}
\begin{tabular}{| l | cccc |}
    \hline
    $Z$ & $e$ & $x_1$ & $x_2$ & $x_3$\\
    \hline
    $(e,k_1,k_2,k_3)$ & $(1,1,1)$ & $1$ & $1$ & $1$\\
    \hline
\end{tabular}

\bigskip

\bigskip 

\bigskip

\begin{tabular}{| l | cccc |}

    \hline
    $Z$ & $e$ & $x_1$ & $x_2$ & $x_3$\\
    \hline
    $(w_1,w_1+k_1,w_1+k_2,w_1+k_3)$ & $-1$ & $(-1,-1,-1)$ & $1$ & $1$\\

    $(w_2,w_2+k_1,w_2+k_2,w_2+k_3)$ & $-1$ & $1$ & $(-1,-1,-1)$ & $1$\\   

    $(w_3,w_3+k_1,w_3+k_2,w_3+k_3)$ & $-1$ & $1$ & $1$ & $(-1,-1,-1)$\\   
    \hline

\end{tabular}
\end{table}

\bibliographystyle{alpha}
\bibliography{mainbib}

\end{document}